\documentclass[11pt]{amsart}

\usepackage{amsmath,amsthm,amsfonts,amssymb,mathrsfs}
\usepackage{inputenc,mathrsfs}

\numberwithin{equation}{section}
\usepackage[dvips]{graphicx}
\usepackage[colorlinks]{hyperref}

\newtheorem{definition}{Definition}[section]
\newtheorem{theorem}{Theorem}[section]
\newtheorem{proposition}{Proposition}[section]
\newtheorem{lemma}{Lemma}[section]
\newtheorem{corollary}{Corollary}[section]

\newtheorem{remark}{Remark}[section]

\newcommand{\be}{\begin{equation}}
\newcommand{\ee}{\end{equation}}

\newcommand{\leftn}{|}
\newcommand{\rightn}{|}
\newcommand{\leftdn}{||}
\newcommand{\rightdn}{||}

\newcommand{\bnd}{{W^{\lambda,\lambda/2}_p(\partial M_T)}}

\newcommand{\sob}{{W^{2,1}_p(M_T)}}
\newcommand{\sobb}{{W^{2,1}_p(M_T)}}

\newcommand{\Lpcal}{{\mathcal L_p^{\lambda,\frac{\lambda}{2}}(\partial M_T)}}
\newcommand{\Lpcall}{{\mathcal L_p^{\lambda,\frac{\lambda}{2}}(V_T)}}
\newcommand{\Lpcallab}{{\mathcal L_p^{\alpha,\beta}(V_T)}}
\newcommand{\Lpcallx}{{\mathcal L_p^{\alpha,0}(V_T)}}
\newcommand{\Lpcallt}{{\mathcal L_p^{0,\beta}(V_T)}}
\newcommand{\Lsob}{{L^{2,1}_p(M_T)}}

\DeclareMathOperator{\dive}{\mathrm{div}}
\DeclareMathOperator{\Ric}{\mathrm{Ric}}
\DeclareMathOperator{\Rm}{\mathrm{Rm}}
\DeclareMathOperator{\tr}{\mathrm{tr}}

\DeclareMathOperator{\dete}{\mathrm{det}}

\title[The Ricci flow on manifolds with boundary]{The Ricci flow on manifolds with boundary}

\author[]{Panagiotis Gianniotis}
\address{Department of Mathematics\\
University College London\\
25 Gordon Street, London WC1E 6BT\\}
\email{p.gianniotis@ucl.ac.uk}

\begin{document}
\begin{abstract}
We study the short-time existence and regularity of solutions to a boundary value problem for the Ricci-DeTurck equation on a manifold with boundary. Using this, we prove the short-time existence and uniqueness of the Ricci flow prescribing the mean curvature and conformal class of the boundary, with arbitrary initial data. Finally, we establish that under suitable control of the boundary data the flow exists as long as the ambient curvature and the second fundamental form of the boundary remain bounded.
\end{abstract}

\maketitle

\section{Introduction}
The aim of this paper is to study the deformation for a short period of time of a Riemannian metric $g^0$ on a compact Riemannian manifold with boundary using the Ricci flow
\begin{eqnarray}
  \partial_t g=-2\Ric(g), \label{eq2}
\end{eqnarray}
which was introduced by Hamilton in \cite{Ham1}. He established the short-time existence and uniqueness of solutions with $g(0)=g^0$ and used it to study three-dimensional manifolds admitting metrics with positive Ricci curvature. Later on, Shi in \cite{Shi1} proved the short time existence of the flow for complete manifolds with uniformly bounded Riemann tensor. Ever since, it has been proven to be a valuable tool in the study of the interaction between geometry and topology, providing a natural geometric deformation of Riemannian manifolds.

A natural question to ask is whether one can deform the geometry of a manifold with boundary using the Ricci flow, and what would be appropriate boundary conditions. The obstacle, as in the case without boundary, is the diffeomorphism invariance of the Ricci tensor which is why the equation is not  parabolic. One needs to solve a modified parabolic equation first, as DeTurck did in \cite{DeTurck} and then relate its solution to the Ricci flow. In the case of manifolds with boundary though, the challenge is to impose boundary conditions that on the one hand will lead to a parabolic boundary value problem for the modified equation, and at the same time tie well with the geometric character of the Ricci flow. 

The first work in this direction was by Y.Shen in \cite{YShen}, where he established a short-time existence result for compact manifolds with umbilic boundary. Moreover, he extended Hamilton's result in \cite{Ham1} to the case of manifolds with totally geodesic boundary. The convex (and umbilic) case was studied later by Cortissoz in \cite{Cortissoz}. However, one would like to deform more arbitrary metrics than in \cite{YShen}. To this direction, Pulemotov in \cite{Pulemotov} proved a short-time existence result for manifolds with boundary of merely constant mean curvature.

More work has been done on the two-dimensional Ricci flow, and the closely related Yamabe flow. Both have been studied under Neumann-type boundary conditions. See for instance the contributions of Brendle in \cite{Brendle1}, \cite{Brendle2}, \cite{Brendle3}, Tong Li in \cite{Tong} and Cortissoz in \cite{Cort}. Also, Giesen and Topping in \cite{GiesTopExistence} study the Ricci flow on general incomplete surfaces from a different point of view. They show existence of solutions which become instantaneously complete for positive time and completely classify their asymptotic behaviour. Moreover, Topping in \cite{TopUniqueness} shows that such flows depend uniquely on the initial data.

Heuristically,  the Ricci flow is closely related to the corresponding ``elliptic" problem, the Einstein equations. Boundary value problems for Einstein metrics have been studied for instance by Anderson in \cite{Anderson1},  Anderson and Khuri in \cite{AndKhuri}, Schlenker in \cite{Schlenker} and Reula in \cite{reula}.  In particular, in \cite{Anderson1} it is shown that the conformal class and the mean curvature of the boundary give elliptic boundary conditions for the Bianchi-gauged Einstein equation. Notice that in the case of three-dimensional manifolds with boundary, solving such a boundary value problem also gives rise to  immersions of the boundary data (conformal class and mean curvature) in the canonical simply connected spaces of constant curvature. We refer the reader to \cite{Anderson2} for details on this point of view. A parabolic approach may provide further understanding of these geometric problems.

A solution to the Ricci flow is not expected to be determined uniquely by the mean curvature only, as in \cite{Pulemotov}, which hints that it should be supplemented with additional boundary data. In the following, we study boundary value problems for the Ricci-DeTurck flow and the  Ricci flow, under the boundary conditions proposed in \cite{Anderson1}. The main result of this study is a local existence and uniqueness result for the Ricci flow on manifolds with boundary. To the knowledge of the author, this is the first result which allows the flow to start from an arbitrary initial metric.

The methods used can also be applied to study boundary value problems for geometric flows related to static metrics in General Relativity (see \cite{AndKhuri}). However, we won't pursue this direction here, as we plan to discuss it in a future paper.

Let $M^{n+1}$ be a compact $n+1$ dimensional manifold with boundary $\partial M$ and interior $M^o$. If $g$ is a smooth Riemannian metric on $M$ we will denote by $\mathcal H(g)$ the mean curvature of the boundary and by $u^T$ the part of the tensor $u$,  tangential to the boundary. Moreover, if $\gamma$ is some Riemannian metric on $\partial M$, let $[\gamma]$ be its conformal class, namely 
$$[\gamma]=\left\{ \gamma'=\phi^2\gamma \textrm{\;, for all positive functions $\phi$ on $\partial M$}\right\}.$$

Now, let $g^0$ be an arbitrary smooth Riemannian metric on $M$, $\gamma(x,t)\in C^\infty(\partial M\times [0,+\infty))$ a smooth time-dependent family of metrics on $\partial M$ and a function  $\eta(x,t)\in C^\infty (\partial M\times [0,+\infty))$. We assume that they satisfy the zeroth order compatibility conditions
\begin{eqnarray}
\begin{array}{ccc}
\mathcal H (g^0)&=&\eta|_{t=0}\label{zeroorder}\\
\left[(g^0)^T\right]&=&[\gamma|_{t=0}]. 
\end{array}
\end{eqnarray}

Moreover, let $\kappa>0$ be the constant defined in Section \ref{RDT_proof}, which bounds the $W^{2,p}$ norm of $g^0$ and appropriate H\"older norms of $\gamma$ and $\eta$. For the precise definitions of the function spaces used and the restrictions on the values of $p,\epsilon$ and $\alpha$ appearing below the reader is invited to consult Section \ref{notation}. 

\begin{theorem}\label{Ricci-DTrk}
Let $M^{n+1},g^0,\gamma,\eta$ be as above. Consider an arbitrary family of background metrics $\tilde g\in C^{\infty}(M\times [0,+\infty))$ which satisfies in addition the zeroth order compatibility condition $\tilde g(0)=g^0$. Take $K>0$ and set
 $$\Lambda=\max\left\{\kappa, \sup_t \left\{ \leftdn \tilde g(t) -g^0\rightdn_{W^{2,p}(M^o)}+\leftdn \partial_t \tilde g(t)\rightdn_{L_p(M^o)}\right\}\right\}.$$
Then, there exists a $T>0$ which depends only on $\Lambda>0$ and $K>0$ and a unique solution $g(t)$, $t\in [0,T]$, of the Ricci-DeTurck equation, 
\be
\partial_t g=-2\Ric(g)+\mathcal L_{\mathcal W(g,\tilde g)} g,\label{eq1}
\ee
where $\mathcal W(g,\tilde g)_l=g_{lr}g^{pq}(\Gamma(g)^r_{pq}-\Gamma(\tilde g_t)^r_{pq})$, satisfying on $\partial M\times [0,T]$ the boundary conditions:
\begin{eqnarray}
\mathcal W(g,\tilde g)&=&0,\label{rdtbdata1}\\
\mathcal H(g)&=&\eta,\label{rdtbdata2}\\
\left[g^T\right]&=&[\gamma].\label{rdtbdata3}
\end{eqnarray}
and the estimate $\leftdn g-g^0 \rightdn_\sob\leq K$. The solution is $C^\infty$ away from the corner $\partial M\times 0$, and extends on $M\times [0,T]$ as a $C^{1+\alpha,\frac{1+\alpha}{2}}$ family of metrics. Moreover, if the data $g^0$, $\gamma$, $\eta$ and $\tilde g$ satisfy the necessary higher order compatibility conditions (see Section \ref{comp}), then $g$ is $C^{k+\alpha,\frac{k+\alpha}{2}}$  up to $\partial M\times 0$.
\end{theorem}

Now, Theorem \ref{Ricci-DTrk} allows us to prove in Section \ref{rflow} a short-time existence result for the Ricci flow on an arbitrary compact Riemannian manifold with boundary. Here, the existence time of the flow is controlled in terms of bounds on the geometry of the initial data.

\begin{theorem}\label{Ricci_flow}
Let $g^0$, $\gamma$, $\eta$ as in Theorem \ref{Ricci-DTrk}, and suppose 
\begin{eqnarray}
\sup_M |\Ric(g^0)|_{g^0}+\sup_{\partial M}|\Ric((g^0)^T)|_{g^0}&\leq& C, \label{ex1}\\
i_{g^0},  i_{(g^0)^T} , i_{b,g^0} &\geq& C^{-1},\label{ex2} \\
\textrm{diam}(M,g^0)&\leq& C, \label{ex3}\\
|\gamma|_{1+\epsilon,\frac{1+\epsilon}{2}}+|\gamma^{-1}|_0+\sup_{\partial M\times 0}|R(\gamma)|+\left|\eta \right|_{\epsilon,\frac{\epsilon}{2}}&\leq& C, \label{ex4}\\
C^{-1}\gamma|_{t=0}\leq (g^0)^T &\leq& C\gamma|_{t=0} \label{ex5}
\end{eqnarray}
for some $C>1$. Then, there exists a smooth solution  $g(t)$ to (\ref{eq2}), for $0<t\leq T$, that satisfies on $\partial M\times (0,T]$ the boundary conditions (\ref{rdtbdata2})-(\ref{rdtbdata3}) and $T>0$ depends only on $C$.

Moreover, as $t\searrow 0$, $g(t)$ converges in the $C^{1,\alpha}$ Cheeger-Gromov sense to $g^0$ and $C^\infty$ away from the boundary. Namely, there exist a smooth family of diffeomorphisms $\phi_t$ of $M$, $t>0$, such that $\phi_t^* g(t)\rightarrow g^0$.

Also, if $g^0, \gamma , \eta$ satisfy the necessary higher order compatibility conditions for the Ricci tensor $\Ric$ to be in the class $C^k(\overline{M_T})$  (see Section \ref{rflow}), then
\begin{enumerate}
\item As $t\searrow 0$, $g(t)$ converges to $g^0$ in the $C^{k+2,\alpha}$ Cheeger-Gromov sense.
\item $g\in C^k(\overline{M_T})\cap C^{\infty}(M^o\times [0,T])$, and there exists a $C^{k+1}$ diffeomorphism $\phi$  of $M$ which fixes the boundary and is $C^\infty$ in the interior such that $g(0)=\phi^* g^0$. Also, if $k\geq 1$, $\phi$ is $C^{k+2}$ and $g\in C^{k+1}(\overline{M_T})$. 
\item The Riemann tensor $\Rm$ is in $C^k(\overline{M_T})$ and $\Rm(g(0))=\phi^*\Rm(g^0)$.
\end{enumerate}
\end{theorem}

Here, $i_{g^0},  i_{(g^0)^T}$ denote the injectivity radii of $(M,g^0)$, $(\partial M,(g^0)^T)$ respectively and $i_{b,g^0}$ denotes the ``boundary injectivity radius\rq\rq{}, namely the maximal size of the collar neighbourhood of $\partial M$ in which the normal exponential map from the boundary is a local diffeomorphism. See Definition \ref{inj_radii}. Also, we write $R(\gamma)$ for the scalar curvature of $\gamma$.

We note that a version of Theorem \ref{Ricci_flow} in which the initial data are obtained in the usual sense, namely $g(0)=g^0$, does hold. However, such a solution will generally not be $C^\infty$ smooth up to the boundary even for positive time. This issue is related to the invariance of the equation under diffeomorphisms and is discussed in Remark \ref{bad2}.

We prove Theorem \ref{Ricci-DTrk} in Section \ref{dtrf} with a fixed-point argument,  following the method of Weidemaier in \cite{MR1163209} and applying Solonnikov's work on linear parabolic systems under general boundary conditions in \cite{MR0211083}. The main advantage compared to an implicit function theorem approach is that the study of the nonlinearities of the equation and the boundary conditions allows us to obtain uniform control on the existence time. 

Note that the control of the existence time obtained in Theorem \ref{Ricci-DTrk} does not tie well with the geometric nature of the Ricci flow, mainly because it involves norms which depend on the choice of the background smooth structure and metric. From this point of view, Theorem \ref{Ricci_flow} is more satisfactory, as the lower bound on the existence time depends only on the geometry of the initial data $g^0$ and norms of the boundary conditions. 

It is well known that incomplete solutions of the Ricci flow are in general not unique. On a manifold with boundary though, the boundary data (\ref{rdtbdata2})-(\ref{rdtbdata3}) allow us to obtain the following uniqueness result.
\begin{theorem}\label{uniqueness}
A  solution to the boundary value problem (\ref{eq2}),(\ref{rdtbdata2})-(\ref{rdtbdata3}) in $C^3(\overline M_T)$ is uniquely determined by the initial data $g^0$ and the boundary data $([\gamma],\eta)$.
\end{theorem}

Theorems \ref{Ricci-DTrk}, \ref{Ricci_flow} and \ref{uniqueness} generalize to Theorem \ref{general}, where $\eta$ also depends on the metric $g^T$ induced on the boundary by $g(t)$.

Finally, in Sections \ref{example} and \ref{ext} we move towards the study of more global issues. In Section \ref{example} we demonstrate the necessity of the bound on the boundary injectivity radius in Theorem \ref{Ricci_flow}. We construct examples with flat initial data and uniformly controlled boundary conditions whose existence time becomes arbitrarily small. This is quite surprising, since on closed manifolds a curvature bound suffices to prevent such behaviour.  Section \ref{ext} is devoted in the proof of the following theorem, which is a continuation principle for the Ricci flow on manifolds with boundary.
\begin{theorem}\label{extend}
Let $g(t)$, $t\in [0,T)$ be a smooth ($C^{\infty}$) Ricci flow on $M$ with smooth boundary data $([\gamma(t)],\eta(t))$  defined for $t\in [0,T')$. Suppose $T<\infty$ be the maximal time of existence and $T<T'$. 
Then
$$\sup_{0\leq t < T} \left( \sup_{x\in M} |\Rm(g(t))|_{g(t)}+\sup_{x\in \partial M}|\mathcal A(g(t))|_{g(t)} \right) =+\infty.$$
\end{theorem}
\noindent\textbf{Acknowledgements:} The author would like to thank his adviser Michael Anderson for suggesting this problem and for valuable discussions and comments.

\section{Notation, definitions, background material}\label{notation}
Let $M^{n+1}$ be a smooth, compact, $n+1$ dimensional manifold with boundary $\partial M$, and interior $M^o$. We will use the notation $M_T=M^o\times (0,T)$, $\partial M_T=\partial M \times (0,T)$. 

\subsection{Function Spaces.}
We need to define the function spaces we will use. First, fix a smooth Riemannian metric $h$ on $M$ and denote by $\widehat\nabla$ its Levi-Civita connection. We also need to fix an open cover $\{U_s\}$ of $M$, and a collection of charts $\phi_s$ such that
\begin{eqnarray}
\phi_s: U_s\rightarrow B(0,1)\subset \mathbb R^{n+1} &&\textrm{\quad, if $U_s$ does not intersect the boundary}\nonumber\\
\phi_s: U_s\rightarrow B(0,1)^+\subset \mathbb R^{n+1}&&\textrm{\quad, if $U_s$ intersects the boundary.}\nonumber
\end{eqnarray}
In the last case assume that $\left.\phi_s\right|_{\partial M\cap U_s}:\partial M\cap U_s\rightarrow V:=B^n(0,1)\subset \mathbb R^n$. We will use the convention that Greek indices correspond to directions tangent to the boundary, counting from $1$ to $n$. Moreover, $\rho_s$ will be a partition of unity subordinate to that open cover. 

Consider any tensor bundle $E$ of rank $k$ over $ M$, with projection map $\pi$, equipped with the connection inherited by $\widehat\nabla$.  The completion of the space of the time dependent $C^\infty(M_T)$ sections of $E$ with respect to the  norm
$$\leftdn u\rightdn_\sob=\leftdn u\rightdn_{L_p(M_T)}+\leftdn \widehat\nabla u\rightdn_{L_p(M_T)}+\leftdn \widehat\nabla^2 u\rightdn_{L_p(M_T)}+\leftdn \partial_t u \rightdn_{L_p(M_T)}$$
will be denoted by $\sobb$. Let also $$\leftn u \rightn_{L^{2,1}_p(M_T)}=\leftdn \partial_t u \rightdn_{L_p(M_T)}+\leftdn \widehat\nabla^2 u\rightdn_{L_p(M_T)}.$$
If $\tau$ is a section of $E$, we will denote by ${}^s\tau^{ijk...}_{\mu\nu...}$ the coordinates of this tensor with respect to the trivialization based at $U_s$.

We define the following norm for time dependent $L_p$ sections of   $E_{\partial M}=\left\{v\in E\left| \pi\circ v\in \partial M\right.\right\}$ and for $\lambda=1-1/p$:
$$\leftdn v\rightdn_\bnd=\leftdn v\rightdn_{L_p(\partial M_T)}+\leftn v\rightn_\Lpcal$$
Here, setting $\hat\rho_s=\rho_s\circ\phi_s^{-1}$, we define
$$\leftn v\rightn_\Lpcal=\sum_s\max_{i_1,...,i_k}\leftn\hat\rho_s{}^sv^{i_1,...,i_l}{}_{i_{l+1},...,i_k}\rightn_\Lpcall$$
where, for every function $f\in L_p(V_T)$ 
\begin{eqnarray}
\leftn f\rightn_\Lpcallab^p&=&\leftn f\rightn_\Lpcallx^p+\leftn f\rightn_\Lpcallt^p\nonumber\\
\leftn f\rightn_\Lpcallx^p&=&\sum_{\mu=1}^n \int_0^{+\infty}h^{-(1+p\alpha)}\leftdn \Delta_{\mu,h}f\rightdn_{L_p(V_{\mu,h,T})}^p dh\nonumber\\
\leftn f\rightn_\Lpcallt^p&=&\int_0^{+\infty}h^{-(1+p\beta)}\leftdn \Delta_{t,h}f\rightdn_{L_p(V_{T-h})}^p dh.\nonumber
\end{eqnarray}
In the above, 
\begin{eqnarray}
\Delta_{\mu,h}f(y,t)&=&f(y+he_\mu,t)-f(y,t)\nonumber\\
\Delta_{t,h}f(y,t)&=&f(y,t+h)-f(y,t)\nonumber\\
V_{\mu,h,T}&=&\left\{(y,t)\in V_T| y+he_\mu\in V\right\}.\nonumber
\end{eqnarray}
Analogous spaces exist also in the elliptic setting, see for instance \cite{Sol}.

For $l > 0$ nonintegral, we will denote by $C^{l,l/2}(M\times [0,\tau],E)$ the Banach space of time dependent sections $u$ of $E$ having continuous up to the boundary derivatives $\partial_t^r\widehat\nabla^q u$ for all $r,q$ satisfying $2r+q<l$, satisfying appropriate H\"{o}lder conditions in the time and space directions. More precisely, the norm is given by

$$|u|_{l,l/2}=\sup_{s}\max_{I} |{}^s u_{I}|_{[l],B(0,1)}+\sup_s\max_I \left<{}^s u_{I}\right>_{l,l/2,B(0,1)},$$
where ${}^s u_I$ are the coordinate functions of $u$ in the coordinate system $U_s$ and 

\begin{eqnarray}
|f|_{k,B(0,1)}&=&\sum_{0\leq 2r+q\leq k} ||\partial^r_t\partial_x^q f||_\infty \nonumber\\
\left<f\right>_{l,l/2,B(0,1)}&=& \sum_{2r+q=[l]}\left< \partial_t^r \partial_x^q f \right>_{l-[l],x}+\sum_{0<l-2r-q<2}\left< \partial_t^r \partial_x^q f \right>_{\frac{l-2r-q}{2},t}.\nonumber
\end{eqnarray}
Here, for $0<\rho<1$
\begin{eqnarray}
\left< f \right>_{\rho,x}&=&\sup_{x\neq y,\;t}\frac{|f(x,t)-f(y,t)|}{|x-y|^\rho}\nonumber\\
\left< f \right>_{\rho,t}&=&\sup_{t\neq t',\;x}\frac{|f(x,t)-f(x,t')|}{|t-t'|^\rho}.\nonumber
\end{eqnarray}
We will also denote by $|u|_k$ and $\left< u \right>_{l,l/2}$ the norms
\begin{eqnarray}
\left< u \right>_{l,l/2}&=&\sup_s\max_I \left<{}^s u_{I}\right>_{l,l/2,B(0,1)}\nonumber\\
|u|_k &=& \sup_{s}\max_{I} |{}^s u_{I}|_{k,B(0,1)}.\nonumber
\end{eqnarray}
For any integer $k\geq 0$ we will denote by $C^{k}(M\times[0,\tau])$ the space of sections with all the derivatives $\partial_t^r\widehat\nabla^{q} u$ for $2r+q\leq k$ continuous, equipped with the norm $|\cdot|_k$.

By the definition of $\bnd$ it is not hard to see that $C^{\epsilon,\frac{\epsilon}{2}}(\partial M_T)$ embeds in $\bnd$, provided that $\epsilon>\lambda$. We will also need the following embedding theorems.
\begin{lemma}\label{emb1}
\mbox{}
\begin{enumerate}
\item For $1<p<\infty$, and $u\in W_p^{2,1}(M_T)$,  $$||\widehat\nabla u||_{W_p^{\lambda,\lambda/2}(\partial M)}\leq C_1 ||u||_{W_p^{2,1}(M_T)}.$$
\item If $\frac{n+3}{2}<p<\infty$ and $0<\alpha<\min(1,2-(\frac{n+3}{p}))$, then $$\left<u\right>_{\alpha,\alpha/2}\leq C_2\left(\delta^{2-(n+3)/p-\alpha}\leftn u \rightn_\Lsob+\delta^{-(n+3)/p-\alpha}\leftdn u\rightdn_{L^p(M_T)}\right).$$
\item If $n+3<p<\infty$ and $0<\alpha\leq1-(n+3)/p$, then
$$\langle\widehat\nabla u\rangle_{\alpha,\alpha/2}\leq C_3\left(\delta^{1-(n+3)/p-\alpha}\leftn u \rightn_\Lsob+\delta^{-(1+(n+3)/p+\alpha)}\leftdn u\rightdn_{L^p(M_T)}\right).$$
\end{enumerate}
In the above, the constants do not depend on $T>0$ and $0<\delta\leq\min(d,T^{1/2})$, where $d$ is a constant depending on the chosen atlas $\{U_s\}$.
\end{lemma}

\begin{proof}
See Lemma 3.3 at Chapter II of \cite{MR0241822} or Lemma A.1 in \cite{MR1163209}.
\end{proof}

From now on let us fix some $p>n+3$ and some $\alpha\leq 1-\frac{n+3}{p}$. Then, as the previous Lemma implies, the Sobolev space $\sob$ embeds in the H\"{o}lder space $C^{1+\alpha,\frac{1+\alpha}{2}}(\overline{M_T})$. Moreover, we get the following estimates (see Corollary A.2  in \cite{MR1163209}).

\begin{lemma}\label{emb2}
For all $u\in W_p^{2,1}(M_T)$, with $u(.,0)\equiv 0$, $n+3<p<\infty$, $0<\gamma=(1-\frac{n+3}{p})/2$ and all sufficiently small $T>0$
\begin{enumerate}
\item $|u|_1\leq C_4 T^\gamma |u|_{L_p^{2,1}(M_T)}.$
\item $|u|_{\mathcal{L}_p^{\beta,\beta/2}(\partial M_T)}\leq C_5 T^\gamma |u|_{L_p^{2,1}(M_T)}$, for all $\beta\in(0,1).$
\end{enumerate}
\end{lemma}

Also, we will be using the following product estimate.
\begin{lemma}\label{product}
If $f,g\in\Lpcallab\cap L^{\infty}(V_T)$ and $\hat\rho=\rho\circ\phi^{-1}$, then 
$$|\hat\rho f g|_\Lpcallab\leq C_6||fg||_\infty+||f||_\infty|\hat\rho g|_{\mathcal L_p^{\alpha,\beta}(V_T)}+||g||_\infty|\hat\rho f|_{\mathcal L_p^{\alpha,\beta}(V_T)}.$$
\end{lemma}

\subsection{The mean curvature.}
Let $g$ be a Riemannian metric on $M$ and $N$ the ourward unit normal to $\partial M$ with respect to $g$. The second fundamental form $\mathcal A$ of the boundary is defined by
$$\mathcal A=\frac{1}{2}\left( \mathcal L_N g\right)^T.$$
The mean curvature of the boundary with respect to the metric $g$ is then given by $$2\mathcal H(g)=\tr_{g^T} \mathcal L_N g.$$
In the following we are going to need the following formulae, which can be proven by direct computation.

\begin{lemma}
If $g_t$ is a smooth one-parameter family of metrics, such that $g_0=g$, and $\partial_t g|_{t=0}\equiv h$, the first variation of the mean curvature of the boundary is given by the formula:
\be
2\mathcal H_g'=\tr_{g^T} \nabla_N h +2\delta_{\partial M}\left(h(N)^T\right)-h(N,N)\mathcal H(g).\nonumber
\ee
\end{lemma}

\begin{lemma}\label{local}
In the local coordinates defined in this section the mean curvature of the boundary of $M$ is given by
$$2\mathcal H=g^{T,\alpha\beta}\nu^i\partial_i(g_{\alpha\beta})-\left( \frac{2g^{0l}g^{\alpha k}}{\sqrt{g^{00}}}+\frac{g^{0l}g^{0k}g^{0\alpha}}{(\sqrt{g^{00}})^3}-\frac{g^{T,\alpha\beta}g_{0\beta}g^{0l}g^{0k}}{\sqrt{g^{00}}}\right)\partial_\alpha(g_{kl}).$$
\end{lemma}

\section{A linear parabolic initial-boundary value problem.}
Let $g$ be a $C^{1+\epsilon}$ Riemannian metric on $M$, for some $\epsilon>1-\frac{1}{p}$, $\gamma$ the induced metric on the boundary, $\beta_g=\dive_g -\frac{1}{2}d\tr_g$ the Bianchi operator and $\mathcal H'_g$ be the linearization of the mean curvature at $g$.

We will also denote by $W^{2,1}_{p,0}$ and $W^{\lambda,\lambda/2}_{p,0}$ the subspaces whose elements satisfy the initial condition $u|_{t=0}=0$.
\begin{theorem}\label{linear}
Consider the following linear parabolic initial-boundary value problem on symmetric 2-tensors on $M$
$$\partial_t u-\tr_g \widehat\nabla^2 u= F(x,t),$$
\be\left.\begin{array}{rcl}
\beta_g(u)&=& G(x,t)\\
\mathcal H'_g(u)&=&D(x,t)\\
u^T-\frac{\tr_\gamma u^T}{n}\gamma &=&0\\
\end{array}\right\}\textrm{on $\partial M$,}\label{L}
\ee
$$u|_{t=0}=u_0,$$
for $F\in L^p(M_T)$, $G, D$ in the corresponding $W^{\lambda,\lambda/2}_{p}(\partial M_T)$ space and $u_0\in W^{2,p}(M^o)$. Assuming that the zeroth order compatibility conditions
\begin{eqnarray}
 \beta_g(u_0)&=&G(x,0),\nonumber\\
\mathcal H'_g(u_0)&=&D(x,0),\nonumber\\
u_0^T-\frac{\tr_\gamma u_0^T}{n}\gamma &=&0\nonumber
\end{eqnarray}
hold, problem (\ref{L}) has a unique solution $u\in W^{2,1}_{p}(M_T)$ which satisfies the estimate
\begin{eqnarray}
\leftdn u \rightdn_\sob\leq C_8\left( \leftdn F\rightdn_{L^p(M_T)}+\leftdn G \rightdn_\bnd+\leftdn D \rightdn_\bnd+\leftdn u_0 \rightdn_{W^{2,p}(M^o)}\right).
\end{eqnarray}
Moreover the constant $C_8$ stays bounded as $T\rightarrow 0$ and depends on the $C^{1+\epsilon}$ norms of $g$ and $g^{-1}$.
\end{theorem}

\begin{proof}
The method followed in Chapter IV of \cite{MR0241822} and Theorem 5.4 of \cite{MR0211083} carries over to the manifold setting, after the necessary adaptation to the realm of manifolds and vector bundles (see \cite{Pulemotov}). We only need to show that the following boundary value problem on $\mathbb R^{n+1}_+=\{x^0\geq 0\}\subset \mathbb R^{n+1}$ satisfies the complementing condition (see \cite{MR0241822},\cite{MR0211083} and \cite{MR0252806}).  
\begin{eqnarray}
\partial_t u_{kl} - \Delta_{eucl}u_{kl}=\widehat F_{kl} \qquad\qquad\qquad\textrm{ on } \mathbb R^{n+1}_+, \nonumber\\
\nonumber\\
\begin{array}{ccc}
\delta^{ij}\partial_i(u_{jk})-\frac{1}{2}\delta^{ij}\partial_{k}u_{ij}&=&\widehat G_k\nonumber\\
\delta^{\alpha\beta}\partial_o u_{\alpha\beta}-2\delta^{\alpha\beta}\partial_{\alpha}u_{\beta 0}&=&\widehat D\label{P}\\
u_{\alpha\beta}-\frac{\delta^{\epsilon\zeta}u_{\epsilon\zeta}}{n}\delta_{\alpha\beta}&=&0\nonumber
\end{array}
\qquad\qquad\textrm{ on } \{x^0=0\},\nonumber
\end{eqnarray}
and
\begin{eqnarray}
u|_{t=0}=0.\nonumber
\end{eqnarray}
Here, $\widehat F_{kl}\in L_p(\mathbb R^{n+1}_+)$ and $\widehat G_k, \widehat D\in W^{\lambda,\lambda/2}_{p,0}(\partial \mathbb R^{n+1}_+)$. One obtains (\ref{P}) by expressing (\ref{L}) in local coordinates around a point $x$ of the boundary, with $g_{ij}(x)=\delta_{ij}$, freezing the coefficients at $(x,0)$ and keeping the higher order terms. 
The principal symbols of the boundary operators are:
\begin{eqnarray}
i\sum_l \xi_l h_{lk}-\frac{i}{2}\sum_l \xi_k h_{ll}           \label{symbol1}\\
i\xi_0\sum_\alpha h_{\alpha\alpha}-2i\sum_\alpha\xi_{\alpha}h_{0\alpha} \label{symbol2}
\end{eqnarray}
and the principal symbol of the parabolic operator $\partial_t-\Delta_{eucl}$, is $(p+|\zeta|^2+\tau^2)h_{ij}$, where $\zeta=(\zeta_1,\ldots,\zeta_n)\in\mathbb R^n$ and $|\zeta|$ its Euclidean norm. We obtain the following positive root $\hat\tau=i\sqrt{p+|\zeta|^2}$.
Setting equations $(\ref{symbol1}),(\ref{symbol2})$ to zero and letting $\xi_0=\hat\tau$, $\xi_\alpha=\zeta_\alpha$, we get the following system:
\begin{eqnarray} 
i\hat\tau h_{00}+i\sum_\alpha \xi_{\alpha}h_{\alpha 0}-\frac{i}{2}\hat\tau\sum_l h_{ll}&=&0 \label{a1}\\
i\hat\tau h_{0\mu}+i\sum_\alpha\xi_\alpha h_{\alpha\mu}-\frac{i}{2}\xi_\mu\sum_l h_{ll}&=&0 \label{a2}\\
i\hat\tau\sum_\alpha h_{\alpha\alpha}-2i\sum_\alpha \xi_\alpha h_{0\alpha}&=&0 \label{a3}\\
h_{\alpha\beta}&=&\phi \delta_{\alpha\beta}.\label{a3b}
\end{eqnarray}
Since the principal symbol of the equation is in diagonal form, the complementing condition is equivalent to proving that system (\ref{a1})-(\ref{a3b}) has only the zero solution when $(p,\zeta)$ satisfy
\be
\textrm{Re}p\geq-\delta_1|\zeta|^2\label{assumption}
\ee
for some $0<\delta_1<1$.

From equation $(\ref{a1})$ we have \be 2i\sum_\alpha \zeta_\alpha h_{\alpha 0}=i\hat\tau\sum_l h_{ll}-2i\hat\tau h_{00}=i\hat\tau(\tr h-2h_{00})\label{a4},\ee
while multiplying equation  $(\ref{a2})$ by $2\zeta_\mu$ and then adding over $\mu$ we find:
\be
\sum_\mu 2i\hat\tau\zeta_\mu h_{0\mu}+2i\sum_{\alpha,\mu}\zeta_\alpha\zeta_\mu h_{\alpha\mu}-i\sum_\mu \zeta_\mu^2\tr h = 0. \label{a5}
\ee
This gives, taking (\ref{a4}) and $h_{\alpha\mu}=\phi\delta_{\alpha\mu}$ into account:
\be
i\hat\tau^2(\tr h-2h_{00})+2i|\zeta|^2\phi-i|\zeta|^2\tr h=0 \label{a6}
\ee
which, after substituting for $\hat\tau$, leads to the equation:
\be
p h_{00}=p n \phi+2(n-1)|\zeta|^2\phi\label{maestros}
\ee
Now, by equation $(\ref{a3})$ we have:
\be 2i\sum_\alpha\zeta_\alpha h_{0\alpha}=i\hat\tau\sum_\alpha h_{\alpha\alpha}=i\hat\tau\phi n \ee
which combined with $(\ref{a1})$ gives:
\be
2i\hat\tau h_{00}+i\hat\tau\phi n -i\hat\tau\tr h=0\\
\ee
and therefore $i\hat\tau h_{00}=0$. Now, (\ref{assumption}) implies that $p\neq -|\zeta|^2$, which gives $\hat\tau\neq 0$ and thus $h_{00}=0$.

Now, by (\ref{maestros}) we have that 
\be
\phi\left(pn+2|\zeta|^2(n-1)\right)=0.\label{maestro2}
\ee
However, assumption (\ref{assumption}) implies that $\left(pn+2|\zeta|^2(n-1)\right)\neq 0$,  since $\frac{2(n-1)}{n}\geq 1$  for $n\geq 2$.  This gives that $\phi=0$.

Now we have established that $\phi=h_{00}=0$ it is easy to see that $h_{0\mu}=0$, by (\ref{a3}). This proves the complementing condition for system (\ref{P}).
\end{proof}
\begin{remark}\label{evolve}
Theorem \ref{linear} is still valid if we consider $\gamma_t$ and $g_t$ evolving such that $\gamma_0=g^T$. Note that the complementing condition is satisfied if $\gamma_t$ and $g_{t}^T$ are in the same conformal class. If not, the openness of this condition implies that it holds at least for some short time $\hat \tau>0$ depending on  $C^{\epsilon,\epsilon/2}$ bounds of $\gamma_t$ and $g_t$.  Thus, we either get local (in time) existence or a global solution and the constant $C_8$ depends on the $C^{1+\epsilon,\frac{1+\epsilon}{2}}$ norms of $g_t$ and $\gamma_t$ and the $C^0$ norms of $g^{-1}$ and $\gamma^{-1}$.
\end{remark}
\section{A boundary value problem for the Ricci-DeTurck flow.}\label{dtrf}

Let $g^0$ be a $C^2$ Riemannian metric on $M^{n+1}$. Consider also  $\gamma(x,t)\in C^{1+\epsilon,\frac{1+\epsilon}{2}}(\partial M_T)$, a family of boundary metrics and a function $\eta(x,t)\in C^{\epsilon,\frac{\epsilon}{2}}(\partial M_T)$, where $\epsilon$ is always $1-\frac{1}{p}<\epsilon<1$ and $p>n+3$.  Moreover, assume the zeroth order compatibility conditions (\ref{zeroorder}) hold. 

We supplement the Ricci flow equation
\be
\partial_t g=-2\Ric(g),\label{ricciflow}
\ee
with the boundary conditions
\begin{eqnarray}
\left[g^T\right]&=&\left[\gamma_t\right],\label{bdata}\\
\mathcal H(g)&=&\eta(x,t),\nonumber
\end{eqnarray}
and the initial condition 
\be 
g(0)= g^0, \label{idata}
\ee
and aim to study the existence and regularity of solutions.

 As is well known, the Ricci flow equation is not strongly parabolic, so we will first study the Ricci-DeTurck equation
\be
\partial_t g=-2\Ric(g)+\mathcal L_{\mathcal W(g,\tilde g)} g,\label{riccideteqn}\\
\ee
with the boundary conditions
\begin{eqnarray}
\mathcal W(g,\tilde g)&=&0, \nonumber\\
\left[g^T\right]&=&\left[\gamma_t\right],\label{bndry}\\
\mathcal H(g)&=&\eta(x,t).\nonumber
\end{eqnarray}
Here,  $\mathcal W(g,\tilde g)_l=g_{lr}g^{pq}\left(\Gamma_{pq}^r(g)-\tilde\Gamma_{t,pq}^r \right)$, $\Gamma(g)$ being the Christoffel symbols of $g$, and $\tilde \Gamma_t$ the Christoffel symbols of a $C^2$ family of metrics $\tilde g_t$ with $\tilde g|_{t=0}= g^0$  (i.e $\tilde g\in C^2(M\times [0,T])$).

\begin{remark}
The geometric nature of Ricci flow requires the boundary data to be geometric, namely invariant under diffeomorphisms that fix the boundary. The data (\ref{bdata}) have this property. However, passing to the DeTurck equation we need to impose the additional, gauge-dependent boundary condition $\mathcal W(g,\tilde g)=0$.
\end{remark}
\begin{remark}
We allow the background metric $\tilde g_t$ to vary and define a time dependent reference gauge. This, as will be discussed in Section \ref{comp}, allows higher regularity of the solution on $\partial M\times 0$.
\end{remark}
\subsection{Short-time existence of the Ricci-DeTurck flow.}\label{RDT_proof}
We can now state and prove the main short time existence Theorem. First, define
\begin{eqnarray}
 \kappa=\max\left\{\leftdn g^0 \rightdn_{W^{2,p}(M^o)},  |g^0|_{1+\epsilon},|(g^0)^{-1}|_0, |\gamma|_{1+\epsilon,\frac{1+\epsilon}{2}},\leftn \eta -\eta_0\rightn_{\epsilon,\frac{\epsilon}{2}}\right\}.\nonumber
\end{eqnarray}
Then, the following theorem holds.
\begin{theorem}\label{rdt}
Consider the boundary value problem $(\ref{riccideteqn})$,$(\ref{bndry})$ with initial condition $g(0)=g^0$. For the data $(g^0,\tilde g,\eta,\gamma)$ define
$$\Lambda= \max\left\{\kappa,\sup_t \left\{ \leftdn \tilde g -g^0\rightdn_{W^{2,p}(M^o)}+\leftdn \partial_t \tilde g(t)\rightdn_{L_p(M^o)}\right\}\right\}.$$
For any $K>0$ there exists a $T=T(\Lambda,K)>0$ and a solution  $g(t)\in\sob$ of this initial-boundary value problem which satisfies  $\leftdn g-g^0 \rightdn_\sob\leq K$.
\end{theorem}
\begin{proof}
Using the background connection $\widehat\nabla$ the Ricci-DeTurck equation (\ref{riccideteqn}) can be expressed as
\begin{eqnarray}\partial_t g-\tr_g \widehat\nabla^2_{.,.}g=\mathcal R(g(x,t),\widehat\nabla g(x,t))-\mathcal L_{V(g)} g,\end{eqnarray}
where $V(g)=g_{ir}g^{pq}(\widetilde\Gamma^r_{t,pq}-\widehat\Gamma^r_{pq})$, while in local coordinates we get (see \cite{MR2274812})
\begin{eqnarray}
\mathcal R(g,\widehat\nabla g)_{ij}&=&g^{pq}h^{kl}\left(g_{ik}\widehat R_{jplq}+g_{jk}\widehat R_{iplq}\right)\nonumber\\
&&-g^{pq}g^{kl}\left(\frac{1}{2}\widehat\nabla_i g_{kp}\widehat\nabla_j g_{lq} +\widehat\nabla_p g_{jk}\widehat\nabla_l g_{iq}-\widehat\nabla_p g_{jq}\widehat\nabla_q g_{il}\right)\nonumber\\
&&+g^{pq}g^{kl}\left(\widehat\nabla_j g_{kp}\widehat\nabla_q g_{il}+\widehat\nabla_i g_{kp}\widehat\nabla_q g_{jl}\right).\nonumber
\end{eqnarray}
Here $h$ and $\widehat\nabla$ are the background metric and connection we used to define the function spaces.

Moreover, we will express the boundary condition for the conformal class in the form:
$$g_t^T-\frac{\tr_{\gamma_t}g_t^T}{n}\gamma_t=0.$$
Following \cite{MR1163209}, for $K, T>0$  we define the following subset of $\sob$:
$$M_K^T(g^0)=\left\{u\in\sob\left|u|_{t=0}=g^0,\;\leftdn u-g^0\rightdn_\sob\leq K \right.\right\}.$$
Choose $\delta>0$ such that $(g^0)^{ij}\xi_i\xi_j\geq \delta |\xi|^2_{eucl}$ in every coordinate system of the fixed atlas. Note that $\delta$ is controlled from below in terms of $\kappa$. Lemma $\ref{emb2}$ implies that for every $K>0$, there exists $0<T_o(K,g^0)\leq 1$ such that $\dete(u_{ij})\geq \delta/2$ and $(u^{-1})^{ii}\geq\delta/2$ for every $u\in M_K^{T_o}(g^0)$. In particular, $u(x,t)$ is a metric for all $t\in[0,T_o]$.

Now, let $T\leq T_o$. For every $w\in M_K^T(g^0)$ the following linear parabolic boundary value problem is well defined:

\begin{eqnarray}
\partial_t u-\tr_{g^0}\widehat\nabla^2 u&=&\mathcal R(w(x,t),\widehat\nabla w(x,t))\\&&-\mathcal L_{V(w)} w-\tr_{g^0}\widehat\nabla^2 w+\tr_w \widehat\nabla^2 w\equiv F_w,\nonumber\\
\beta_{g^0}(u)&=&\beta_{g^0}(w)-\mathcal W(w)\equiv D_w,\nonumber\\
\mathcal H_{g^0}'(u)&=&\mathcal H_{g^0}'(w)-\mathcal H(w)+\eta (x,t) \equiv G_w,\nonumber\\
u^T-\frac{\tr_{\gamma_t}u^T}{n}\gamma_t&=&0.\nonumber\\
\nonumber\\
u|_{t=0}&=&g^0.\nonumber
\end{eqnarray}
and has a unique solution $u\in\sob$, by Theorem \ref{linear}. This defines a map
$$S:M_K^T(g^0)\rightarrow \sob,$$
where $S(w)$ is this solution.

Notice that a fixed point of $S$ solves the nonlinear boundary value problem. Therefore, it suffices to prove that $S$ is a map from $M_K^T(g^0)$ to itself and also a contraction, as long as $T$ is small enough. The existence of the fixed point will follow, since $M_K^T(g^0)$ is a complete metric space.

It is easy to see that $\sigma=S(w)-g^0$ satisfies
\begin{eqnarray}
\partial_t \sigma-\tr_{g^0}\widehat\nabla^2\sigma&=&F_w+\tr_{g^0}\widehat\nabla^2  g^0\equiv \widehat F_w,\nonumber\\
\nonumber\\
\beta_{g^0}(\sigma)&=&D_w,\nonumber\\
\mathcal H_{g^0}'(\sigma)&=&\mathcal H_{g^0}'(w-g^0)-(\mathcal H(w)-\mathcal H(g^0))\\&&\qquad\qquad+\eta (x,t)-\eta(x,0) \equiv \widehat G_w,\nonumber\\
\sigma^T-\frac{\tr_{\gamma_t} \sigma^T}{n}\gamma_t &=&0,\nonumber\\
\nonumber\\
\sigma|_{t=0}&=&0.\nonumber
\end{eqnarray}
Here we used that $\beta_{g^0}(g^0)=0$ and the compatibility condition  $\mathcal H(g^0)=\eta|_{t=0}$. Lemma \ref{selfmap} below and the parabolic estimate of Theorem \ref{linear} show that for any $K$, $S$ maps $M_K^T(g^0)$ to itself, if $T$ is small enough.

Finally, for any $w_1, w_2\in M_K^T(g^0)$, $S(w_1)-S(w_2)$ similarly satisfies a linear initial-boundary value problem of the form (\ref{linear}). Then, the estimate of Lemma \ref{lipconst} below shows that $S$ is a contraction for small $T>0$. 

The uniform bound of existence time follows from the fact that a uniform bound of $\kappa$ implies uniform bounds of the constants of Lemmata \ref{selfmap}, \ref{lipconst}, and the constant $C_8$ of the parabolic estimate of Theorem \ref{linear}.
\end{proof}

\subsubsection{Lemmata \ref{selfmap} and \ref{lipconst}}
\begin{lemma}\label{selfmap}
Let $w\in M_K^T(g^0)$ for some $K>0$ and $T\leq T_o(K,g^0)$. Then, there exists a constant $C(K,\tilde g,\eta)$ and a function $\zeta:[0,+\infty)\rightarrow [0,+\infty)$ with $\zeta(T)\rightarrow 0$ as $T\rightarrow 0$, such that the following estimate holds:
\be
\leftdn \widehat F_w\rightdn_{L_p(M_T)}+\leftdn D_w\rightdn_\bnd+\leftdn \widehat G_w\rightdn_\bnd\leq C(K,\tilde g, \eta)\zeta(T),\nonumber
\ee
where  $C(K,\tilde g,\eta)= C\left(K,\sup_t \left\{\leftdn \tilde g_t- g^0 \rightdn_{W^{2,p}(M^o)}+\leftdn \partial_t\tilde g_t \rightdn_{L_p(M^o)}\right\},  \leftdn g^0 \rightdn_{W^{2,p}(M^o)},|\eta-\eta_0|_{\epsilon,\frac{\epsilon}{2}}\right)$.

\end{lemma}
\begin{proof}
Since $w\in M_K^T(g^0)$, Lemma \ref{core2} below implies that $w\in C^1(\overline{M_T})$ and therefore
$$ |\mathcal R(w,\widehat\nabla w) |_h  \leq C\left(K,\leftdn g^0 \rightdn_\sob\right).$$
This gives
$$ \leftdn \mathcal R(w,\widehat\nabla w) \rightdn_{L_p(M_T)} \leq C\left(K,\leftdn g^0 \rightdn_{W^{2,p}(M^o)}\right) \zeta(T).$$
 Next, we estimate 
\begin{eqnarray}
\leftdn\mathcal L_{V(w)} w\rightdn_{L_p(M_T)} &\leq& T^{1/p}\sup_t\leftdn \mathcal L_{V(w)}w|_t \rightdn_{L_p(M^o)} \nonumber\\ &\leq& C\left(K, \sup_t \leftdn \tilde g_t \rightdn_{W^{2,p}(M^o)}\right)\zeta(T)\nonumber.
\end{eqnarray}
We also have
$$\leftdn \tr_{g^0}\widehat\nabla^2 g^0\rightdn_{L_p(M_T)}\leq C\left(\leftdn g^0 \rightdn_{W^{2,p}(M^o)}\right)T^{1/p}.$$
Combining these estimates we obtain 
 $$\leftdn \mathcal R(w,\widehat\nabla w)-\mathcal L_{V(w)} w+\tr_{g^0}\widehat\nabla^2 g^0\rightdn_{L_p(M_T)}\leq C\left(K,\sup_t \leftdn\tilde g(t) \rightdn_{W^{2,p}(M^o)}\right)\zeta(T).$$
To estimate the rest of $F_w$ we estimate using Lemma \ref{core2}:
\begin{eqnarray}
\leftn(( g^0)^{ij}-w^{ij})\widehat\nabla^2_{i,j}w\rightn_h&\leq& C\max_{k,l}\leftn((g^0)^{ij}-w^{ij})\widehat\nabla^2_{i,j}w_{k,l}\rightn\nonumber\\
&\leq& C\left(K,\leftdn g^0\rightdn_\sob\right) \max_{i,j,k,l}\leftn\widehat\nabla^2_{i,j}w_{k,l}\rightn\nonumber\\
&\leq& C\left(K, \leftdn g^0\rightdn_\sob\right)\zeta(T) \leftn\widehat\nabla^2 w\rightn_{h}.\nonumber
\end{eqnarray}
This gives the estimate $$\leftdn \tr_{g^0}\widehat\nabla^2 w-\tr_w\widehat\nabla^2 w \rightdn_{L_p(M_T)}\leq C\left(K,\leftdn g^0\rightdn_{W^{2,p}(M^o)}\right)\zeta(T)$$ and proves
$$\leftdn \widehat F_w \rightdn_{L_p(M_T)}\leq C\left(K, \sup_t \leftdn\tilde g(t)\rightdn_{W^{2,p}(M^o)}\right)\zeta(T).$$
It now remains to control the norms of $\widehat G_w$ and $D_w$. Given any $w\in M_K^T(g^0)$ (we assume that $T<T_o(K, g^0)$), define $h=w- g^0$, and for every $0\leq s\leq 1$
\begin{eqnarray} 
	g_s(x,t)&=&g^0(x)+s\cdot h(x,t)\nonumber
\end{eqnarray}
Then by the fundamental theorem of calculus we get that
$$2\mathcal H(w)-2\mathcal H(g^0)=\int_0^1 2\mathcal H'_{g_s}(h)ds$$
and therefore
\begin{eqnarray}
	\widehat G_w&:=&2\mathcal H'_{ g^0}(h)-\left(2\mathcal H(w)-2\mathcal H( g^0 )\right)+2(\eta(x,t)-\eta(x,0))\nonumber\\
	&=&\int_0^1 [2\mathcal H'_{g_0}(h)-2\mathcal H'_{g_s}(h)]ds +2(\eta(x,t)-\eta(x,0)).\nonumber
\end{eqnarray}
Now, denoting $A_s:=2\mathcal H'_{g_0}(h)-2\mathcal H'_{g_s}(h)$ we calculate
\begin{multline}
	A_s=\underbrace{tr_{g_0^T}(\nabla_{0,N_0}h)-tr_{g_s^T}(\nabla_{s,N_s}h)}_{\alpha_1^s}\\+\underbrace{2\delta_{0,\partial M}(h(N_0)^T)-2\delta_{s,\partial M}(h(N_s)^T)}_{\alpha_2^s}
	+\underbrace{h(N_s,N_s)\mathcal H(g_s)-h(N_0,N_0)\mathcal H(g_0)}_{\alpha_3^s} \label{As}
\end{multline}
which, in the coordinates of the fixed atlas, are
\begin{eqnarray}
	\mathcal H(g)&=&\frac{1}{2}\left(g^{\alpha\beta}\nu^i\partial_i(g_{\alpha\beta})+2\partial_\alpha(\nu^\alpha)+2g^{\alpha\beta}g_{0\beta}\partial_\alpha(\nu^0)\right)\nonumber\\	
	\alpha_1^s&=&(g_0^{\alpha\beta}\nu_0^i-g_s^{\alpha\beta}\nu_s^i)\partial_i h_{\alpha\beta}-2(g_0^{\alpha\beta}\nu_0^{i}\Gamma^l_{o,i\alpha}-g_s^{\alpha\beta}\nu_s^{i}\Gamma^l_{s,i\alpha})h_{l\beta} \nonumber\\
	\alpha_2^s&=&\left(g_s^{\alpha\beta}\nu_s^i-g_o^{\alpha\beta}\nu_o^i\right)\partial_\alpha h_{i\beta}+
\left(g_s^{\alpha\beta}\partial_\alpha(\nu_s^i)-g_o^{\alpha\beta}\partial_{\alpha}(\nu_o^i) \right)h_{i\beta}\nonumber\\
	&&+\left(g_o^{\alpha\beta}\nu_o^i\bar\Gamma^j_{o,\alpha\beta}-g_s^{\alpha\beta}\nu_s^i\bar\Gamma^j_{s,\alpha\beta}\right)h_{ij}\nonumber \\
	\alpha_3^s&=&\frac{1}{2}\left\{
	\left(\nu_s^i\nu_s^j\nu_s^k g_s^{\alpha\beta}\partial_k (g_{s,\alpha\beta})-\nu_0^i\nu_0^j\nu_0^k g_0^{\alpha\beta}\partial_k (g_{0,\alpha\beta})\right)h_{ij}\right.\nonumber\\	&&+2\left(\nu_s^i\nu_s^j\partial_\alpha(\nu_s^\alpha)-\nu_0^i\nu_0^j\partial_\alpha(\nu_0^\alpha)\right)h_{ij}\nonumber\\
&&	+2\left(\nu_s^i\nu_s^j g_s^{\alpha\beta}g_{s,0\beta}\partial_{\alpha}(\nu_s^0)-
	\nu_0^i\nu_0^j g_0^{\alpha\beta}g_{0,0\beta}\partial_{\alpha}(\nu_0^0)\right)h_{ij},\nonumber	
\end{eqnarray}
where $N_s=\nu_s^i\partial_i=-\frac{g_s^{0i}}{(g_s^{00})^{1/2}}\partial_i$ is the outward unit normal, $\bar\Gamma$ the Christoffel symbols of the connection induced on $\partial M$, and $g_s^{ij}$ represents the inverse of the matrix $g_{s,ij}$(i.e the induced metric on the cotangent bundle). To simplify notation, $g_s^{\alpha\beta}$ denotes the inverse of the matrix $\{g_{\alpha\beta}\}_{\alpha,\beta=1,...,n}$.\\

Now, to indicate how the estimates of this lemma are established, we show how the term
\be (g_0^{\alpha\beta}\nu_0^i-g_s^{\alpha\beta}\nu_s^i)\partial_i h_{\alpha\beta}=\left[(g_0^{\alpha\beta}-g_s^{\alpha\beta})\nu_0^i+(\nu_0^i-\nu_s^i)g_s^{\alpha\beta}\right]\partial_i h_{\alpha\beta} \nonumber\ee
is estimated. We have
\begin{eqnarray}
\leftn \hat\rho(g_0^{\alpha\beta}-g_s^{\alpha\beta})\nu_0^i\partial_ih_{\alpha\beta}\rightn_\Lpcall
&\leq& C\leftn g_0^{\alpha\beta}-g_s^{\alpha\beta}\rightn_0 \leftn \partial_ih_{\alpha\beta} \rightn_0 \leftn \nu_0^i\rightn_0 \nonumber\\
&&+\leftn g_0^{\alpha\beta}-g_s^{\alpha\beta}\rightn_0 \leftn \nu_0^i\rightn_0\leftn \hat\rho\partial_ih_{\alpha\beta}\rightn_\Lpcall\nonumber\\
&&+\leftn g_0^{\alpha\beta}-g_s^{\alpha\beta}\rightn_0 \leftn \partial_ih_{\alpha\beta}\rightn_0\leftn \hat\rho\nu_0^i\rightn_\Lpcall\nonumber\\
&&+\leftn \nu_0^i\rightn_0 \leftn \partial_ih_{\alpha\beta}\rightn_0\leftn \hat\rho( g_0^{\alpha\beta}-g_s^{\alpha\beta})\rightn_\Lpcall\nonumber\\
&\leq& C\left(K,\rightdn g^0\leftdn_\sob\right)\zeta(T) \leftdn h\rightdn_\sob,\nonumber
\end{eqnarray}
where the last inequality follows from Lemma \ref{core2}. 

The terms that are of zeroth order in $h$, for example $2(g_0^{\alpha\beta}\nu_0^{i}\Gamma^l_{o,i\alpha}-g_s^{\alpha\beta}\nu_s^{i}\Gamma^l_{s,i\alpha})h_{l\beta}$, are of first order in $g_0$ and $g_s$, but they are estimated in a similar way:
\begin{multline}
2(g_0^{\alpha\beta}\nu_0^{i}\Gamma^l_{o,i\alpha}-g_s^{\alpha\beta}\nu_s^{i}\Gamma^l_{s,i\alpha})h_{l\beta}=\\
2\left[(g_0^{\alpha\beta}-g_s^{\alpha\beta})\nu_0^i \Gamma_{0,i\alpha}^l+g_s^{\alpha\beta}(\nu_0^i-\nu_s^i)\Gamma_{0,i\alpha}^l+g_s^{\alpha\beta}\nu_s^i(\Gamma_{0,i\alpha}^l-\Gamma_{s,i\alpha}^l)\right]h_{l\beta}.\nonumber
\end{multline}

For example, the term $g_s^{\alpha\beta}\nu_s^i(\Gamma_{0,i\alpha}^l-\Gamma_{s,i\alpha}^l) h_{l\beta}$ can be estimated again using Lemma \ref{core2};
\begin{eqnarray}
\leftn \hat\rho g_s^{\alpha\beta}\nu_s^i(\Gamma_{0,i\alpha}^l-\Gamma_{s,i\alpha}^l) h_{l\beta}\rightn_\Lpcall &\leq&C\leftn g_s^{\alpha\beta}\nu_s^i\rightn_{0}\leftn \Gamma_{0,i\alpha}^l-\Gamma_{s,i\alpha}^l\rightn_{0}\leftn h_{l\beta}\rightn_0\nonumber\\
&+&\leftn g_s^{\alpha\beta}\nu_s^i \rightn_0  \leftn h_{l\beta}\rightn_0 \leftn \hat\rho(\Gamma_{0,i\alpha}^l-\Gamma_{s,i\alpha}^l)\rightn_\Lpcall\nonumber\\
&+&\leftn g_s^{\alpha\beta}\rightn_0 \leftn h_{l\beta}\rightn_0 \leftn \Gamma_{0,i\alpha}^l-\Gamma_{s,i\alpha}^l \rightn_0\leftn \hat\rho\nu_s^i\rightn_\Lpcall\nonumber\\
&+&\leftn \nu_s^i \rightn_0 \leftn h_{l\beta}\rightn_0 \leftn \Gamma_{0,i\alpha}^l-\Gamma_{s,i\alpha}^l \rightn_0 \leftn\hat\rho g_s^{\alpha\beta} \rightn_\Lpcall\nonumber\\
&+&\leftn \nu_s^ig_s^{\alpha\beta}\rightn_0 \leftn \Gamma_{0,i\alpha}^l-\Gamma_{s,i\alpha}^l \rightn_0 \leftn \hat\rho h_{l\beta}\rightn_\Lpcall\nonumber\\
&\leq& C\left(K,\leftdn g^0\rightdn_\sob\right)\zeta(T)\leftdn h\rightdn_\sob.\nonumber
\end{eqnarray}

The procedure indicated above carries over to estimate all the terms of $A_s$, providing us with the estimate
\begin{eqnarray}
\leftn A_s \rightn_{0,V_T} +\leftn \hat\rho A_s \rightn_\Lpcall &\leq&  C\left(K,\leftdn g^0\rightdn_\sob\right)\zeta(T)\leftdn h\rightdn_\sob\nonumber\\
&\leq&  C\left(K,\leftdn g^0\rightdn_\sob\right)\zeta(T),\nonumber
\end{eqnarray}
since $\leftdn h \rightdn_\sob\leq K$. Now, under the assumptions for $\eta$, this proves that 
\be \leftdn \widehat G_w \rightdn_\bnd\leq  C\left(K,\leftdn g^0\rightdn_\sob, \left|\eta-\eta_0\right|_{\epsilon,\frac{\epsilon}{2}}\right)\zeta(T).\nonumber \ee

Similarly, the linearization of the map $w\mapsto\mathcal W(w,\tilde g)$ at $u(x,t)$ is given by
\be 
	\mathcal W'_u(\tau)_l=\beta_u(\tau)_l+(\tau_{lr}u^{pq}-u_{lr}\tau_{ij}u^{ip}u^{jq})(\Gamma(u)_{pq}^r-\tilde\Gamma_{t,pq}^r).\label{VFlinearization}
\ee
So, given any $w\in M_K^T( g^0)$, $T<T_0$, since $\beta_{g^0}( g^0)=0$ we have
\begin{eqnarray}
	(D_w)_l&=&\beta_{g^0}(w)_l-(\mathcal W(w,\tilde g)-\mathcal W(g^0,\tilde g))_l-\mathcal W(g^0,\tilde g)_l\nonumber\\
&=&\beta_{g^0}(h)_l-\int_0^1 \mathcal W'_{g_s}(h)_l ds-\mathcal W(g^0,\tilde g)_l\nonumber\\
	&=&\int_0^1 (\beta_{g^0}(h)-\beta_{g_s}(h))_l ds+\nonumber \\ &&\int_0^1 (h_{lr}g_{s}^{pq}-g_{s,lr}h_{ij}g_s^{ip}g_s^{jq})(\Gamma_{s,pq}^r-\tilde\Gamma_{t,pq}^r) ds -\mathcal W(g^0,\tilde g)_l.\nonumber
\end{eqnarray}
	
Again, using a coordinate system intersecting the boundary, we have:
\begin{eqnarray}
\beta(h)&=&g^{ij}\left(\partial_i h_{jl}-h_{rl}\Gamma_{ij}^r-h_{jr}\Gamma_{il}^r\right)-\frac{1}{2}\partial_l(g^{ij}h_{ij})\nonumber\\
\beta_{g^0}(h)_l-\beta_{g_s}(h)_l&=&( g^{0,ij}-g_s^{ij})\partial_i h_{jl}-(g^{0,ij}\Gamma_{0,ij}^r-g_s^{ij}\Gamma_{s,ij}^r)h_{rl}\nonumber\\&&-(g^{0,ij}\Gamma_{0,il}^r-g_s^{ij}\Gamma_{s,il})h_{jr}\nonumber\\
&&-\frac{1}{2}\left[(\partial_l g^{0,ij}-\partial_l g_s^{ij})h_{ij}+(g^{0,ij}-g_s^{ij})\partial_l h_{ij}\right].\nonumber
\end{eqnarray}
Finally, a series of estimates of the same form as those used for the mean curvature part of the boundary conditions gives the required estimate:
\begin{eqnarray}
\leftdn D_w \rightdn_\bnd&\leq& C\left(K,\leftdn \tilde g-g^0\rightdn_\sob,\leftdn g^0 \rightdn_{W^{2,p}(M^o)}\right)\zeta(T)\nonumber\\
&\leq& C(K,\tilde g,\eta) \zeta(T).\nonumber
\end{eqnarray}
\end{proof}

A similar line of reasoning also proves the following lemma. See also \cite{MR1163209}.
\begin{lemma}\label{lipconst}
Let $K>0$ and $T\leq T_o(K,g^0)$. Then, there exists a constant $C(K,\tilde g)$ such that for every $w_1,w_2\in M_K^T( g^0)$ the following estimate holds:
\begin{multline}
\leftdn F_{w_1}-F_{w_2}\rightdn_{L_p(M_T)}+\leftdn D_{w_1}-D_{w_2}\rightdn_\bnd+\leftdn G_{w_1}-G_{w_2}\rightdn_\bnd \\ \leq C(K,\tilde g)\zeta(T)\leftdn w_1-w_2\rightdn_\sob,
\end{multline}
where 
$C(K,\tilde g)= C\left(K, \sup_t \left\{\leftdn \tilde g- g^0 \rightdn_{W^{2,p}(M^o)}+\leftdn \partial_t\tilde g \rightdn_{L_p(M^o)}\right\}, \leftdn g^0 \rightdn_{W^{2,p}(M^o)}\right).$
\end{lemma}

\subsubsection{Technical Lemmata.}

\begin{lemma}\label{inverse}
Let $\delta_0>0$. There exists a positive constant $C$, such that for matrix valued functions $g, g_l\in L^\infty(V_T)\cap\mathcal L_p^{\alpha,\beta}(V_T)$, $l=1,2$ for which $det(g_{ij}), det(g_{l,ij})\geq\delta_0$ and $g^{ii},g_l^{ii}\geq\delta_0$ holds:
\begin{eqnarray}
|\hat\rho g^{ij}|_\Lpcallab&\leq& C|g|_0\left(\left|\hat\rho g\right|_\Lpcallab + 1\right),\nonumber\\
|\hat\rho\left(g_1^{ij}-g_2^{ij}\right)|_\Lpcallab&\leq& C\cdot B_1\cdot B_2 \left(\left|\hat\rho\right(g_1-g_2\left)\right|_\Lpcallab + |g_1-g_2|_{0}\right),\nonumber \\
|\hat\rho (g^{00})^{-1/2}|_\Lpcallab&\leq& C\cdot|g|_0\left(\left|\hat\rho g\right|_{\mathcal L_p^{\alpha,\beta}} + 1\right),\nonumber\\
|\hat\rho\left((g_1^{00})^{-1/2}-(g_2^{00})^{-1/2}\right)|_\Lpcallab&\leq& C\cdot B_1\cdot B_2.\left(\left|\hat\rho\left(g_1-g_2\right)\right|_\Lpcallab + |g_1-g_2|_{0}\right).\nonumber
\end{eqnarray}
where $|g_l|_0\leq B_1$ and $|\hat\rho g_l|_{\mathcal L^{\alpha,\beta}_p(V_T)}\leq B_2$ and the constants depend on $\delta_0$ and the cutoff function $\hat\rho$.
\end{lemma}
\begin{proof} The result follows from Weidemaier, Corollary A.3. 
\end{proof}

Direct consequence of Lemma \ref{inverse} are the following estimates.
\begin{lemma}\label{core2}
Let $g_0, g_1\in M_K^T(g^0)$, $T\leq T_0(K, g^0)$,  $g_s=g_0+s(g_1-g_0)$ and $(U,\phi,\rho)$ a chart whose domain intersects the boundary, with the corresponding cutoff function $\rho$ and $\hat\rho=\rho\circ\phi^{-1}$. Let also  $\nu_s^i=-\frac{g_s^{0i}}{(g_s^{00})^{-1/2}}$ be the components of the outward unit normal to the boundary with respent to $g_s$. Then
\begin{enumerate}
\item $\leftn g_{s,ij}\rightn_{1,V_T}+|\hat\rho g_{s,ij}|_\Lpcall+|\hat\rho\partial_k g_{s,ij}|_\Lpcall\leq C(K, g^0)$.

\item $\leftn g_0^{ij}-g_s^{ij}\rightn_{1}+\leftn g_{0,ij}-g_{s,ij}\rightn_{1}\leq C(K,g^0,\delta)\zeta(T)\leftdn h\rightdn_\sob$.

\item $\leftn \hat\rho(g_0^{ij}-g_s^{ij})\rightn_\Lpcall+\leftn \hat\rho(g_{0,ij}-g_{s,ij})\rightn_\Lpcall\leq C(K, g^0,\delta)\zeta(T)\leftdn h\rightdn_\sob$.

\item $\leftn \hat\rho\partial_k g_s^{ij}\rightn_\Lpcall \leq C(K, g^0,\delta)$.

\item $\leftn \nu_s^i \rightn_0 +\leftn \hat\rho \nu_s^i \rightn_\Lpcall+\leftn \hat\rho \partial_\alpha \nu_s^i\rightn_\Lpcall\leq C(K, g^0,\delta)$.

\item $\leftn\nu_0^i-\nu_s^i\rightn_{1}+\leftn\hat\rho(\nu_0^i-\nu_s^i)\rightn_\Lpcall\leq C(K,g^0,\delta)\zeta(T)\leftdn h\rightdn_\sob$.

\item $\leftn\hat\rho\partial_\alpha \nu_s^i \rightn_\Lpcall\leq C(K,g^0,\delta)$.
\end{enumerate}
where $\displaystyle{\lim_{T\rightarrow 0^+}\zeta(T)=0}$.
\end{lemma}
\smallskip

\subsection{Regularity of the Ricci-DeTurck flow.}\label{comp}
The solution of the Ricci-DeTurck boundary value problem obtained in the previous section is in the Sobolev space $\sob$ for $p>n+3$, and therefore has only $C^{1+\alpha,\frac{1+\alpha}{2}}$ regularity in $\overline{M_T}$. In this section we show that certain higher order compatibility conditions on $\partial M$ are necessary and sufficient for higher regularity on $\partial M\times 0$. We also obtain an automatic smoothing effect of the flow for positive time (up to the boundary).

\subsubsection{Higher order compatibility conditions}
Assuming that $g(x,t)$ is a $C^{l+2,\frac{l}{2}+1}(\overline{M_T})$ solution to the Ricci-DeTurck flow 
$$\partial_t g=-2\Ric(g)+\mathcal L_{\mathcal W(g,\tilde g)} g,$$
we easily see that all the derivatives $h_k\equiv\partial_t^k g|_{t=0}\in C^{l+2-2k}(M)$, $0\leq k\leq \left[\frac{l}{2}\right]+1$ are determined by the initial data $g|_{t=0}=g^0\in C^{l+2}(\bar M)$, by differentiating the equation with respect to $t$, and then commuting $\partial_t^k$ with $\partial_i \partial_j$, which is possible as long as $2k+2<l+2$, i.e. $0\leq k\leq \left[\frac{l}{2}\right]$.\\

Moreover, if $g(x,t)$ satisfies the boundary conditions (\ref{bndry}), differentiating  with respect to $t$ we get:
\begin{eqnarray}
\partial_t^{k}\mathcal W(g_t,\tilde g_t)|_{t=0}&=&0,\label{gauge}\\
\partial_t^{k}\mathcal H(g_t)|_{t=0}&=&\partial_t^{k}\eta|_{t=0},\label{mean}\\
\left.\partial_t^{k}\left(g_{t}^T-\frac{\tr_{\gamma_t}(g_{t}^T)}{n}\gamma_t\right)\right|_{t=0}&=&0.\label{conformal}
\end{eqnarray}
So, from (\ref{conformal}) for $k\leq \left[ \frac{l}{2} \right] +1$, we see that additional conditions need to be satisfied by $h_k$, and hence by $g^0$, on $\partial M$. Similarly, for $k\leq \left[\frac{l+1}{2}\right]$ (so as $2k+1<l+2$ ), the $\partial_t^k$ derivatives commute with the space derivatives of the first order operators $\mathcal W, \mathcal H$ on (\ref{gauge}), (\ref{mean}), and give additional restrictions on the initial data on the boundary.\\

In particular, if $l>0$,  since $ (g^0)^T=\gamma|_{t=0}$ we see that $\stackrel{.}{\gamma}|_{t=0}$ is specified, up to a conformal factor by $h_1$:
\be
\stackrel{.}{\gamma}|_{t=0}=h_1^T-\frac{\tr_{\gamma_0} h_1^T}{n}\gamma_0 + f\gamma_0, \label{compatibility_conformal}
\ee
where $f$ is an arbitrary function.
Moreover, if $l>1$, $\dot\eta|_{t=0}$ is also specified by the initial data:
\be 
\dot\eta|_{t=0}=\mathcal H'_{ g^0}(h_1) .\label{compatibility_mean_curvature}
\ee

In the section below, we show how parabolic regularity implies that these conditions are also sufficient to obtain higher regularity of a solution to the Ricci-DeTurck boundary value problem.

\begin{theorem}
Let $g \in \sob$ be a solution to the Ricci-DeTurck boundary value problem (\ref{eq1}),(\ref{rdtbdata1})-(\ref{rdtbdata3}).  Let $l=k+\alpha$, $\alpha\leq 1-\frac{n+3}{p}$. Then the following hold.
\begin{enumerate}
\item (Interior regularity) Suppose that $\tilde g\in C^{l+2,\frac{l+2}{2}}(M^o\times [0,T])$. Then $g\in C^{l+2,\frac{l+2}{2}}(M^o\times [0,T])$. 
\item (Boundary regularity)  If $\eta\in C^{l+1,\frac{l+1}{2}}(\partial M_T)$, $\gamma\in C^{l+2,\frac{l+2}{2}}(\partial M_T)$, $\tilde g\in C^{l+2,\frac{l+2}{2}}(\overline{M_T})$ and the data $g^0$,$\eta$,$\gamma$,$\tilde g$ satisfy the necessary compatibility conditions, then $g\in C^{l+2,\frac{l+2}{2}}(\overline{M_T})$.
\item (Boundary regularity for positive time)   If $\eta\in C^{l+1,\frac{l+1}{2}}(\partial M_T)$,  $\gamma\in C^{l+2,\frac{l+2}{2}}(\partial M_T)$ and $\tilde g\in C^{l+2,\frac{l+2}{2}}(\overline{M_T})$ then, for $\tau\in (0,T)$,  $g\in C^{l+2,\frac{l+2}{2}}(M\times [\tau, T])$.
\end{enumerate}
\end{theorem}

\begin{proof}
Given a Riemannian metric $g_{ij}$ on a domain in $\mathbb R^{n+1}$ we can define the differential operator
$$\mathcal L(\partial_t,\partial_x,g_{ij})(u)_{kl}=\partial_t(u_{kl})-g^{ij}\partial_i\partial_j(u_{kl}),$$
acting on symmetric 2-tensors $u_{kl}$.

In a coordinate system the solution $g(x,t)$ of the Ricci-DeTurck flow satisfies parabolic equation of the form
\be
\mathcal L(\partial_t,\partial_x,g_{ij})(g)_{kl}=\mathcal S(g,\partial g,\tilde g,\partial\tilde g,\partial^2 \tilde g)_{kl}, \label{eq}
\ee
hence, standard interior regularity theory implies Part 1 of the theorem.\\

It remains to study the regularity of $g$ at a neighbourhood of the boundary, under the assumptions of the theorem. \\

We need to establish some notation first. Let $\phi: U\rightarrow \phi(U)\subset\mathbb R^{n+1}$ be any smooth chart on a domain $U$ intersecting the boundary of $M$, such that $\phi(U\cap\partial M)=\phi(U)\cap\{x^0=0\}$, and let $g_{ab}$, $\gamma_{\varepsilon\sigma}$ be symmetric positive definite $(n+1)\times (n+1)$ and $n\times n$ matrices respectively. Define the following differential operators
\begin{eqnarray}
B(\partial_x,g_{ab})(u)_{i}&=&g^{pq}\partial_p(u_{qi})-\frac{1}{2}g^{pq}\partial_i(u_{pq}),\nonumber\\
H(\partial_x,g_{ab})(u)&=&g^{T,\alpha\beta}\nu^i\partial_i(u_{\alpha\beta})+\nonumber\\&& \left(\frac{2g^{0l}g^{\alpha k}}{\sqrt{g^{00}}}-\frac{g^{0l}g^{0k}g^{0\alpha}}{(\sqrt{g^{00}})^3}+\frac{g^{T,\alpha\beta}g_{0\beta}g^{0l}g^{0k}}{\sqrt{g^{00}}}\right)\partial_{\alpha}(u_{kl}),\nonumber\\
C(\gamma_{\varepsilon\sigma})(u)_{\alpha\beta}&=&u_{\alpha\beta}-\frac{\gamma^{\mu\nu}\gamma_{\alpha\beta}}{n}u_{\mu\nu}.\nonumber
\end{eqnarray}

Now, take any $p\in\partial M$, and consider a smooth coordinate system as the above with $\phi(p)=0$ and $g_{ij}(0)|_{t=0}=\delta_{ij}$.

By Lemma \ref{local},  $2\mathcal H(g)=H(\partial_x, g_{ab}(x,t))(g)$. Thus, in addition to (\ref{eq}), $g_{ij}(x,t)$ satisfy the following conditions on $\phi(U)\cap\{x^0=0\}$.  
\begin{eqnarray}
B(\partial_x, g_{ab}(x,t))(g)_i&=&g^{pq}g_{ri}\tilde\Gamma^r_{pq},\nonumber\\
H(\partial_x, g_{ab}(x,t))(g)&=&2\eta(x,t),\label{conditions}\\
C(\gamma_{\varepsilon\sigma})(g)_{\alpha\beta}&=&0,\nonumber
\end{eqnarray}
and the initial condition $g_{ij}|_{t=0}= g^0_{ij}$.\\

Notice that after ``freezing" the coefficients at $x=0,t=0$, the operators $\mathcal L(\delta_{ab})$, $B(\delta_{ab})$, $H(\delta_{ab})$, $C(\delta_{\varepsilon\sigma})$, satisfy the complementing condition, as the computation in Theorem \ref{linear} shows. The openness of this condition implies that  the same is true for the operators $\mathcal L(g_{ab})$, $B(g_{ab})$, $H(g_{ab})$, $C(\gamma_{\varepsilon\sigma})$, as long as $g_{ab}$ and $\gamma_{\varepsilon\sigma}$ are close to $\delta_{ab}$.

We will extend (\ref{eq}), (\ref{conditions}) to a parabolic boundary value problem on $\mathbb R^{n+1}_+$ using a smooth cutoff function $0\leq\eta\leq 1$ on $\mathbb R^{n+1}_+$ supported in  a ball $B^+(0,r)$ such that $\eta|_{B^+(0,r/2)}\equiv 1$. For this, define the metrics
$$a_{ab}(x,t)=\eta g_{ab}(x,t)-(1-\eta)\delta_{ab}$$
on $\mathbb R^{n+1}_+$, and 
$$\alpha_{\varepsilon\sigma}=\eta \gamma_{t,\varepsilon\sigma} +(1-\eta)\delta_{\varepsilon\sigma}$$ 
on $\mathbb R^n$. Then, choosing $r,\tau>0$ small enough, the operators $\mathcal L(a_{ab}(x,t))$, $B(a_{ab}(x,t))$, $H(a_{ab}(x,t))$, $C(\alpha_{\varepsilon\sigma}(x,t))$ will satisfy the complementing  condition, defining a parabolic boundary value problem on $\mathbb R^{n+1}_+\times [0,\tau]$.

Now, let $0\leq\zeta\leq 1$ be a smooth cutoff function supported in $B^+(0,r/2)$, $\zeta|_{B^+(0,r/3)}\equiv 1$, and set $v=\zeta g$.

Then $v$ satisfies the equation
$$\mathcal L(\partial_t,\partial_x,a_{ij})(v)_{kl}=\zeta \mathcal S_{kl}-g^{ij}\partial_i\partial_j(\zeta)g_{kl}-2g^{ij}\partial_i(\zeta)\partial_j(g_{kl})$$
and the boundary conditions
\begin{eqnarray}
B(\partial_x,a_{ab}(x,t))(v)_i&=&\left(1-\frac{n}{2}\right)\partial_i(\zeta)+\zeta g^{pq}\tilde\Gamma^r_{t,pq}g_{ri},\nonumber\\
H(\partial_x,a_{ab}(x,t))(v)&=&2\zeta\eta+n\nu^i\partial_i(\zeta)+\nu^\alpha\partial_\alpha(\zeta)+\frac{g_T^{\alpha\beta}g_{0\beta}g^{0l}g^{0k}g_{kl}}{\sqrt{g^{00}}}\partial_{\alpha}(\zeta),\nonumber\\
C(\alpha_{\varepsilon\sigma})(v)_{\alpha\beta}&=&0\nonumber
\end{eqnarray}
on $\mathbb R^{n+1}_+$. Therefore, by Theorem 5.4 of \cite{MR0211083},  we can prove that $g\in C^{l+2,\frac{l+2}{2}}(\overline M\times [0,\tau])$ for some small $\tau>0$, since $\zeta g^0_{kl}$ satisfies the necessary higher order compatibility conditions as long as $g^0_{kl}$ does.

 At this point we should observe that for $g\in\sob$, the metrics $g(t)$ are uniformly equivalent and satisfy a uniform H\"{o}lder condition in the $t$ direction. Therefore, one can iterate the argument above to show boundary regularity up to time $T$.
 
The regularity up to the boundary for positive time follows in a similar manner, using the local estimates provided by Theorem 5.7 in \cite{MR0211083} (which are the analogous estimates to those in Chapter 4 of \cite{MR0241822}).
\end{proof}

\subsubsection{Regularity of the DeTurck vector field}
According to the following Proposition, the DeTurck vector field $\mathcal W$ can gain one derivative, without requiring all the compatibility conditions needed to increase the regularity of $g$. Only higher order compatibility of the initial data with the reference metrics is needed. Thus, it can be assumed to be as smooth as the solution to the Ricci-DeTurck flow.
\begin{proposition}\label{reg3}
Let $g\in C^{l,l/2}(\overline{M_T})$ be a solution of the Ricci-DeTurck equation, $l>3$. Assume that further that $\tilde g$ is in $C^{l+1,\frac{l+1}{2}}(\overline{M_T})$ and that the compatibility condition $h_k=\partial^k_t\tilde g_t|_{t=0}$ is satisfied for $k\leq\left[\frac{l+1}{2}\right]$. Then, the DeTurck vector field $\mathcal W$ is in $C^{l,l/2}(\overline{M_T})$.
\end{proposition}
\begin{proof}
Applying the Bianchi operator $\beta_g=\dive_g-\frac{1}{2}d\tr_g$ in both sides of the Ricci-DeTurck equation we get
$$\beta_g (\partial_t g)=\beta_g(\mathcal L_{\mathcal W}g).$$
Commuting derivatives we obtain
$$\beta_g(\mathcal L_{\mathcal W}g)=\Delta \mathcal W + \Ric(\mathcal W),$$
where $\Delta =\tr_g \nabla^2$ and $\Ric(\mathcal W)=\Ric(\mathcal W,\cdot)$.

By the linearization formula of $\mathcal W$ (\ref{VFlinearization}) we get that 
$$\partial_t \mathcal W_b=\beta_g(\partial_t g)_b-g_{br}U_{ij}g^{ip}g^{jq}\left(\Gamma_{pq}^r-\tilde\Gamma_{t,pg}^r\right)-g^{pq}\partial_t (\tilde\Gamma_{t,pq}^r),$$
with $U_{ij}=-2\Ric_{ij}+\mathcal L_{\mathcal W}g_{ij}$.

Combining the above we get the following evolution equation for $\mathcal W$
$$\partial_t \mathcal W=\Delta \mathcal W +\Ric(\mathcal W) +Q,$$
where $Q$ is an expression involving at most two derivatives of the metric. By parabolic regularity, given the Dirichlet boundary condition $\mathcal W|_{\partial M}=0$ and the validity of the compatibility conditions at $t=0$ it follows that $\mathcal W\in C^{l,l/2}(M\times [0,T])$ as long as $g\in C^{l,l/2}(M\times [0,T])$.
\end{proof}

\subsection{Uniqueness.}
Let $g_1,g_2\in \sob$ be two solutions to the Ricci-DeTurck boundary value problem (\ref{eq1}),(\ref{rdtbdata1})-(\ref{rdtbdata3}) satisfying the same initial and boundary data. Choosing $K>0$ such that $g_i\in M_K^T(g^0)$, there is a $\hat\tau>0$ such that the map $S$ defined in the proof of the existence Theorem is a contraction map of $M_K^{\hat\tau}(g^0)$ to itself, and therefore has a unique fixed point. Since $g_1$, $g_2$ are both fixed points, they have to agree on $[0,\hat\tau]$. Assuming the data are smooth enough to guarantee that $g_i(t)$ are $C^2$ for $t>0$, one can apply the same argument regarding $t_0$ as initial time. Then, an open-closed argument concludes that $g_1\equiv g_2$ on $[0,T]$.\\

\noindent The results above finally complete the proof of Theorem \ref{Ricci-DTrk}.
\section{The boundary value problem for the Ricci flow}\label{rflow}
Let $ g^0$ be a smooth Riemannian metric on a compact Riemannian manifold with boundary $M$, $\gamma_t$ be a smooth family of metrics of the boundary, and $\eta$ a smooth function on $\partial M\times [0,+\infty)$. We assume that they satisfy the zero order compatibility condition (\ref{zeroorder}). The aim is to study the existence and regularity of a Ricci flow evolution of $ g^0$  on $M$, such that the conformal class of the boundary metric is $[\gamma_t]$ and the mean curvature of the boundary is $\eta$. The existence will follow by the standard argument of pulling back a solution of the Ricci-DeTurck flow by a family of diffeomorphisms. However the issue of how smooth this family is at the corner $\partial M\times 0$ of the parabolic domain will become relevant,  as it may be only $C^0$  despite being smooth everywhere else.  Theorem \ref{Ricci_flow} describes how this phenomenon affects the existence and regularity. Before discussing the proof we make some remarks on the regularity of a solution to the Ricci flow $g(t)$ with the boundary conditions under consideration.

As it was shown in  Subsection \ref{comp}, certain higher order compatibility conditions among the initial and boundary data are necessary for the regularity of the Ricci-DeTurck flow on the corner $\partial M\times 0$. Naturally, such obstruction to regularity appears in any evolution initial-boundary value problem, and so does for the Ricci flow.

For instance, for the boundary value problem (\ref{ricciflow}), (\ref{bdata}), (\ref{idata}), the compatibility conditions (\ref{compatibility_conformal}), (\ref{compatibility_mean_curvature}), with $h_1=-2\Ric(g^0)$, are needed for a $C^2$ or $C^3$ solution to exist. Notice that these compatibility conditions are exclusively formulated in terms of the Ricci tensor of the initial metric. More generally, differentiating the boundary conditions with respect to time,  and using $[\gamma]=[g^T]$, we get
$$
\Ric^T-\frac{\tr_{g^T}\Ric^T}{n}g^T=f\gamma+\frac{\tr_{\gamma}g^T}{n}\dot\gamma.
$$
Also, the mean curvature condition gives,
$$\mathcal H'_g(\Ric)=-\frac{1}{2}\dot\eta.$$
Using the evolution equation of the Ricci tensor under Ricci flow (see for instance \cite{MR2274812}), and the contracted second Bianchi identity we observe that $\Ric$ satisfies the following boundary value problem
\be
\partial_t \Ric=\Delta_L \Ric,\qquad\text{on $M$}\label{Evolution_Ricci}
\ee
where $\Delta_L$ is the Lichnerowicz Laplacian, and on $\partial M$
\begin{eqnarray}
\Ric^T-\frac{\tr_{g^T}\Ric^T}{n}g^T&=&f\gamma+\frac{\tr_{\gamma^T}g^T}{n}\dot\gamma,\nonumber\\
\mathcal H'_g(\Ric)&=&-\frac{1}{2}\dot\eta,\label{boundary_cond_Ricci}\\
\beta_g(\Ric)&=&0.\nonumber
\end{eqnarray}
Notice that the computation in Theorem \ref{linear} shows that it satisfies the complementing condition and  is parabolic. 

Now, since  $\partial^k_t g=-2\partial^{k-1}_t\Ric$, it follows that the compatibility conditions on $g(0)$ needed for $g(t)\in C^k(\overline{M_T})$, with $k> 3$,  are the same as those for $\Ric\in C^{k-2}$ , satisfying (\ref{Evolution_Ricci}),  (\ref{boundary_cond_Ricci}). Note also that by the contracted second Bianchi identity the compatibility conditions of any order hold for the last boundary condition.

For example $\Ric\in C^2(\overline{M_T})$ and $g\in C^4(\overline{M_T})$ require an additional compatibility condition between $\left(\Delta_L\Ric\right)^T$ and $\ddot\gamma|_{t=0}$, which, in the simple case that the conformal class stays fixed along the flow, will be 
$$\left(\Delta_L\Ric(g^0)\right)^T=\rho g^{0,T}$$
for some function $\rho$ on $\partial M$.\\

In the rest of this section we will proceed in the proof of Theorem \ref{Ricci_flow}. However, we will first state and prove the following lemma. This will enable us to provide a lower bound on the existence time of the Ricci flow  in terms of geometric bounds. We first need a few definitions.
\begin{definition}\label{inj_radii}
Let $(M,g)$ be a Riemannian manifold with boundary.
 \begin{enumerate}
  \item We will say that the injectivity radius $i_g$ of $(M,g)$ satisfies the inequality $i_g \geq i_0$ if at any $x\in M\setminus \partial M$ the exponential map restricted to any $B_\rho \subset T_xM$, $\rho\leq \min\{i_0,\frac{1}{2} dist_g (x,\partial M)\}$, is a diffeomorphism onto its image.
  \item The boundary injectivity radius $i_{b,g}$ of $(M,g)$ is the maximal $\rho>0$ such that a $\rho$-tubular neighbourhood of $\partial M$ is diffeomorphic to $\partial M\times [0,\rho)$ via the normal exponential map of $\partial M$.
 \item We will denote by $i_{g^T}$ the injectivity radius of the boundary.
 \end{enumerate}
\end{definition}

\begin{lemma}\label{reparametrize}
 Let $M$ be a compact, smooth, manifold with boundary and fix $\gamma$ a smooth Riemannian metric on $\partial M$. Given any Riemannian metric $g$ on $M$ such that $[g^T]=[\gamma]$ satisfying 
\begin{eqnarray}
\sup_M|\Ric(g)|_g+\sup_{\partial M}|\Ric(g^T)|_{g^T}&\leq& C, \label{curv_bound}\\
i_g,  i_{g^T},  i_{b,g} &\geq& C^{-1}, \\
\textrm{diam}(M,g)&\leq& C, \\
|\gamma|_{1+\epsilon}+ |\gamma^{-1}|_0+\sup_{\partial M} |R(\gamma)|+|\mathcal H(g)|_{\epsilon}&\leq& C, \label{gamma_H_control}\\
C^{-1}\gamma\leq\;\; g^T& \leq &C\gamma \label{uniformity}
\end{eqnarray}
for some $C>1$, there exists $K=K(C)>0$ and a smooth ($C^\infty$) diffeomorphism $\phi$ of $M$ such that $|\phi^* g|_{1+\epsilon}+||\phi^* g ||_{W^{2,p}(M)} \leq K$. Moreover, $\phi|_{\partial M}$ is uniformly controlled in $C^{2,\epsilon}$ in terms of $C$.
\end{lemma}
\begin{proof}
The existence of $K>0$ and $\phi$ such that $\phi^* g$ is $C^{1,\epsilon}$-controlled is a consequence of the $C^{1,\alpha}$-compactness result in \cite{AndTaylor}. To obtain the uniform control of $\phi|_{\partial M}$ in $C^{2,\epsilon}$ it suffices to show that 
$g^T$ is also controlled in $C^{1,\epsilon}$.

By assumption, there exists a smooth function $u$ on $\partial M$ such that $g^T=u^{\frac{4}{n-2}} \gamma$. Here we assume that $n>2$, although the argument for $n=2$ is similar. It is known that $u$ satisfies an elliptic equation of the form 
$$a\Delta u +R(\gamma)u-R(g^T)u^{\frac{n+2}{n-2}}=0,$$
where $a=a(n)$ and $R(\gamma)$, $R(g^T)$ denote the scalar curvatures functions of $\gamma$ and $g^T$ respectively. Also, $\Delta$ denotes the Laplacian with respect to the uniformly equivalent and controlled in $C^{1,\epsilon}$ metric $\gamma$. 

By (\ref{curv_bound}) and (\ref{uniformity}) we obtain uniform bounds on $u$ and $R(g^T)$. Thus, by elliptic regularity $u$ (hence $g^T$) is controlled in $C^{1,\epsilon}$.

 Next, we show that  $h=\phi^*g$ is uniformly bounded in $W^{2,p}(M)$. Set $\bar \gamma=\phi^* \gamma$. Then,  by the uniform $C^{2,\epsilon}$ control of $\phi|_{\partial M}$, (\ref{gamma_H_control}) and (\ref{uniformity}) hold for $h$ and $\bar \gamma$ (maybe with a worse constant $C$). Moreover,  $[h^T]=[\bar \gamma]$, and  $\bar\gamma$, $\mathcal H(h)=\phi^* \eta$ are controlled in $C^{1,\epsilon}$ and $C^{\epsilon}$ respectively. 
 
 The $C^{1,\epsilon}$-control of $h$ implies that harmonic coordinates (or boundary harmonic coordinates in the sence of \cite{AndTaylor}) of $(M,h)$ are $C^{2,\epsilon}$-controlled. Moreover, the harmonic radius is bounded below by \cite{AndTaylor}.
 
 The Ricci tensor becomes elliptic in such coordinates, and standard elliptic regularity provides the interior $W^{2,p}$ control. Similarly, looking at boundary harmonic coordinates, the control of the conformal class of $h^T$ and of the mean curvature $\mathcal H(h)$ of the boundary imply the $W^{2,p}$-control up to the boundary. Note that $\bar \gamma$ is controlled only in $C^{1,\epsilon}$ and not in $C^2$, as required for the $L_p$ estimates in \cite{ADN}. However, the results in \cite{MR0211083} are enough to provide the necessary estimate.
 \end{proof}

\begin{remark}
 The lemma still holds if we replace the mean curvature bound in (\ref{gamma_H_control}) by a bound on $|\eta|_{C^2}$, where $\eta$ is a $C^2$ function such that
\begin{eqnarray}
 \mathcal H(g)=\eta(x,t,g^T,(g^T)^{-1}).\label{mesi_kampylotita}
\end{eqnarray}
This is because, as is demonstrated in the proof above, the bounds (\ref{curv_bound}), (\ref{gamma_H_control}) and (\ref{uniformity}) imply that $g^T$ is controlled in $C^{1,\epsilon}$. Hence, by (\ref{mesi_kampylotita}), we obtain a uniform bound for $|\mathcal H(g)|_\epsilon$.
\end{remark}

\begin{proof}[Proof of Theorem \ref{Ricci_flow}]
We begin by applying Lemma \ref{reparametrize} with $\gamma|_{t=0}$ in place of the metric $\gamma$ and $g^0$ in place of $g$ to obtain a smooth diffeomorphism $\phi$ satisfying the conclusions of that lemma. It suffices to prove the theorem replacing the initial data $g^0$ by $\phi^*g^0$, and the boundary data $\gamma$ and $\eta$ by $\phi^* \gamma$ and $\phi^* \eta$ respectively. To simplify the notation, we will still use $g^0,\gamma,\eta$ to denote these modified initial-boundary data.\\

By Theorem \ref{Ricci-DTrk}, choosing a family of smooth background metrics $\tilde g$, there exists a solution $\hat g(t)$ to the Ricci-DeTurck boundary value problem (\ref{eq1}), (\ref{rdtbdata1})-(\ref{rdtbdata3}), which is in $C^\infty\left(\overline{M_T}-(\partial M\times 0)\right)$ and in $C^{1+\alpha,\frac{1+\alpha}{2}}(\overline{M_T})$ if no other higher order compatibility conditions hold.\\

The DeTurck vector field $\mathcal W(\hat g(t),\tilde g)$ is also in $C^\infty(\overline{M_T}-(\partial M\times 0))$.Then, for some $\varepsilon >0$, the ODE
\begin{eqnarray}
\frac{d}{dt}\psi&=&-\mathcal W\circ\psi, \nonumber\\
\psi_\varepsilon&=&id_M.\label{diffeos}
\end{eqnarray}
defines a unique smooth flow $\psi_t$ for $t>0$, which extends at $t=0$ continuously up to the boundary, and smoothly in the interior.

Then, $g(t)=\psi_t^*\hat g(t)$ solves the Ricci flow equation (see for instance \cite{MR2274812}). Moreover, since the diffeomorphisms $\psi_t$ fix the boundary, and the mean curvature and conformal class are invariant under such diffeomorphisms, it follows that $g(t)$ satisfies the boundary conditions (\ref{rdtbdata2})-(\ref{rdtbdata3}).

Since $\left(\psi_t^{-1}\right)^*g(t)=\hat g(t)$ and $\hat g(t)\rightarrow  g^0$ in the $C^{1,\alpha}$ sense as $t\rightarrow 0$, we get that $g(t)\rightarrow g^0$ in the geometric $C^{1,\alpha}$ sense.

Now, assume that the (unmodified) data $g^0, \gamma, \eta$ satisfy the higher order compatibility conditions necessary for the Ricci tensor to be in $C^k(\overline{M_T})$ and the metric $g$ in $C^{k+2}(\overline{M_T})$ under the Ricci flow. Observe that the modified data $\phi^* g^0, \phi^*\gamma,\phi^*\eta$  satisfy these compatibility conditions too, and to simplify notation we will again denote them by $g^0,\gamma,\eta$.

We need similar compatibility conditions to hold for the Ricci-DeTurck flow, in order to improve the regularity of $\hat g$. In general we don't expect them to hold for an arbitrary choice of background metrics $\tilde g$, so we have to choose them carefully.

As the discusion in Subsection \ref{comp} shows,  the time derivatives at $t=0$ of  solutions $g, \hat g$ to the Ricci flow and Ricci-DeTurck flow respectively
$$h_k=\partial_t^k g|_{t=0},$$
$$\hat h_k=\partial_t^k \hat g|_{t=0},$$
are completely specified by the initial data $g^0$ and the background metrics $\tilde g$ in the case of $\hat h_k$. Observe that if $\tilde g_t$ is chosen so that $\partial_t\tilde g|_{t=0}=0$ and $\partial_t^k\tilde g|_{t=0}=h_k$ for $k>1$, we get
\be\hat h_k=h_k.\label{galign}
\ee
To see this, note that $\hat h_l$ is determined, through the equation, by  $\hat h_0,\ldots,\hat h_{l-1}$ and $\partial_t^k\tilde g|_{t=0}$ for $k<l$. Thus, assuming that  (\ref{galign}) holds for $k<l$ we get
$$\partial_t^k(\mathcal L_{\mathcal W(\hat g,\tilde g)}\hat  g)|_{t=0}=0,\qquad\textrm{for $k<l$}$$
since $\partial_t^k\mathcal W|_{t=0}=0$ for $k<l$. Note that $\partial_t\mathcal W|_{t=0}=0$, by the contracted second Bianchi identity. 
Then, we compute
$$\hat h_l=\partial_t^{l-1}(-2\Ric(\hat g)+\mathcal L_{\mathcal W(\hat g,\tilde g)}\hat  g)|_{t=0}=\partial_t^{l-1}(-2\Ric(\hat g))|_{t=0}=F(\hat h_1,\ldots,\hat h_{l-1})$$
for some expression $F$. On the other hand, 
$$h_l=\partial_t^{l-1}(-2\Ric(g))|_{t=0}=F(h_1,\ldots,h_{l-1})$$
for the same expression $F$. Hence, (\ref{galign}) follows by induction, since $\hat h_0=h_0=g^0$.

Now, (\ref{galign}) implies that higher order compatibility of the data of the Ricci flow boundary value problem imply higher order compatibility of the same order for the Ricci-DeTurck flow.

Theorem \ref{Ricci-DTrk} shows that $\hat g(t)$ is actually in $C^{k+2+\alpha,\frac{k+2+\alpha}{2}}(\overline{M_T})$, which immediately implies that $g(t)$ converges to $g^0$ in the geometric $C^{k+2,\alpha}$ sense.

Moreover, the regularity of $\psi$ in $M\times[0,T]$ is at least $C^{k+1}$, i.e. it has $t$ time and $s$ space derivatives for $2t+s\leq k+1$, since $\mathcal W$ is of first order on the metric. It follows that $\Rm(g(t))=\psi_t^*(\Rm(\hat g(t)))$ is in $C^k(\overline{M_T})$, and $\Rm(g(0))=\psi_0^*\Rm(\hat g(0))$.

By Proposition \ref{reg3}, if $k\geq1$, the DeTurck field and also $\psi$ is in $C^{k+2}$, therefore $g(t)\in C^{k+1}(\overline{M_T})$. Otherwise, if $k=0$, $g(t)\in C^0(\overline{M_T})$ (up to the boundary, at $t=0$).

The lower bound of the existence time $T>0$ is a consequence of the estimates provided by Lemma \ref{reparametrize} and the corresponding estimate for the Ricci-DeTurck flow, after the observation that the background metrics $\tilde g$ can be chosen so that
$$\sup_t\{\leftdn \tilde g(t)-g^0\rightdn_{W^{2,p}(M^o)}+\leftdn \partial_t \tilde g(t)\rightdn_{L_p(M^o)}\}\leq 1.$$
\end{proof}
\begin{remark}\label{bad1}
By parabolic theory,  necessary compatibility conditions are also sufficient to get higher regularity of a solution. However, the Ricci flow is not parabolic, and this is manifested by  loss of derivatives. On the other hand, the Ricci tensor satisfies a parabolic boundary value problem and, as predicted, the compatibility conditions give the expected smoothness.
\end{remark}
\begin{remark}\label{bad2}
Setting the initial condition $\psi|_{t=0}=id_M$  in (\ref{diffeos}), we obtain a solution to the Ricci flow satisfying $g(0)=g^0$. However,  the diffeomorphisms $\psi$  will have finite degree of regularity up to the boundary, even for $t>0$, depending on the compatibility of the data. Thus, $g(t)$  will also have finite regularity along $\partial M\times [0,T]$. This is in contrast to the behaviour of  solutions to parabolic boundary value problems, which become immediately smooth for $t>0$, as long as the boundary data are smooth.

The simple example of a rotationaly symmetric Ricci flow on the $n+1$ dimensional ball illustrates the situation. Consider metrics of the form  
$$g=\phi^2(r) dr^2 + \psi^2(r)ds_{n}^2$$
where $0<r\leq 1$ and $ds_{n}^2$ is the standard metric on $S^{n}$.  Notice that the symmetries imposed fix most of the gauge freedom, allowing only reparametrizations of the radial variable $r$. Under Ricci flow the evolution equations of $\phi$ and $\psi$ are (see \cite{Chow})
\begin{eqnarray}
\partial_t \phi &=& n\frac{\partial_s^2\psi}{\psi}\phi\label{evol1}\\
\partial_t \psi &=& \partial^2_s\psi -(n-1)\frac{1-(\partial_s\psi)^2}{\psi}\label{evol2}
\end{eqnarray}
where $\partial_s=\phi^{-1}\partial_r$. 

The diffeomorphism freedom in the $r$ direction is the reason that $\phi$ does not satisfy a parabolic equation and satisfies a transport-type equation instead. In case the initial and boundary data don't satisfy the first compatibility condition for the mean curvature, $\psi$ will be worse than $C^3(M\times [0,T])$ and the right hand side of (\ref{evol1}) will be worse than $C^1(M\times [0,T])$.  Equation $(\ref{evol1})$ doesn't enjoy the smoothing properties of a parabolic equation and, as an ODE in $t$, we can't expect smooth dependence on the initial data. Thus, low regularity of $g(t)$ at $\partial M\times 0$ can propagate in $\partial M\times \{t>0\}$.
\end{remark}
\smallskip

\subsection{Uniqueness of the Ricci flow.} We can now use the harmonic map heat flow for manifolds with boundary (see \cite{HamBoundary}) to establish the uniqueness of $C^3(\overline M_T)$ solutions to the Ricci flow boundary value problem under consideration, proving Theorem \ref{uniqueness}. Since the overall argument is standard (see \cite{MR2274812}),  we will just point out the necessary modifications to treat the case of manifolds with boundary.

Let $g_1(t), g_2(t)$ be two $C^3(\overline M_T)$ solutions to the Ricci flow satisfying the same initial and boundary conditions. Consider the following heat equations for maps $\phi_i: (M,g_i)\rightarrow (M,g^0)$:
\begin{eqnarray}
\frac{d\phi_i}{dt}&=&\Delta_{g_i(t),g^0}\phi_i\textrm{\quad in $M$},\label{harmonic1}\\
\phi|_{\partial M}&=&id_{\partial M}\textrm{\quad on $\partial M$},\label{harmonic2}
\end{eqnarray}
with initial condition
\be
\phi|_{t=0}=id_{M}.\label{harmonic2}
\ee

For integral $m>0$ and $p>n+3$ we can define the Sobolev spaces $W^{2m,m}_p(M_\varepsilon,M)$ of maps $f:M\rightarrow M$, by requiring the coordinate representations of $f$ with respect to an atlas of $M$ to be in $W^{2m,m}_p(M_\varepsilon)$. This space consists of the $L_p$ functions on $M_\varepsilon=M\times (0,\varepsilon)$ with the derivatives $\partial_t^r\hat\nabla_s$ in $L_p(M_\varepsilon)$ for $2r+s\leq 2m$. The space $W^{2m,m}_p(M_\varepsilon,M)$ does not depend on the atlas used for its definition, as long as $p>n+3$.

The results in Part IV,  Section 11 of \cite{HamBoundary} show that there exist solutions  $\phi_i\in W^{2,1}_p(M_\varepsilon,M)$, for small $\varepsilon>0$. The convexity assumption  of this result for the target $(M,g^0)$ is not needed here, since $\phi|_{t=0}=id_M$ and thus $\phi_i(t)$ remain diffeomorphisms of $M$ for small $t$. Also, by the theory in \cite{MR0241822} and \cite{MR0211083} these results hold under the current assumption for the regularity of $g_i(t)$.

Moreover, the first order compatibility condition for the boundary value problem (\ref{harmonic1})-(\ref{harmonic2}) holds since
 $$\frac{d\phi_i}{dt}=\Delta_{g_i(0),g^0}\phi_i(0)=\Delta_{g^0,g^0}id_M=0.$$
Thus, the diffeomorphisms $\phi_i(t)$  are in $W^{4,2}_p(M_\varepsilon,M)$.

Given the regularity of $\phi_i$ and $g_i$  we know that $\hat g_i=\left(\phi_i(t)\right)_* g_i(t)\in W^{2,1}_p(M_\varepsilon)$. Then, $\hat g_i$ satisfies the Ricci-DeTurck equation with background metric $g^0$ and  the geometric boundary data are still satisfied by $\hat g_i$ since $\phi_i$ fix the boundary. Also notice that the gauge condition 
$$\mathcal W(\hat g_i, g^0)|_{\partial M_T}=0$$ 
holds, since
\be \Delta_{g_i(t),g^0}\phi_i(t)=-\mathcal W(\hat g_i(t),g^0)\circ \phi_i(t)\label{harmcomp}\ee
and 
$$\Delta_{g_i(t),g^0}\phi_i(t)|_{\partial M_T}=0.$$
By the uniqueness of $W^{2,1}_p$ solutions of the Ricci-DeTurck boundary value problem, we have that $\hat g_1(t) = \hat g_2(t)$ and $\mathcal W(\hat g_1(t),g^0)=\mathcal W(\hat g_2(t),g^0)$ for $0\leq t\leq \varepsilon$. Now (\ref{harmcomp}) and (\ref{harmonic1}) imply that $\phi_1=\phi_2$, thus $g_1=\phi_1^*\hat g_1=\phi_2^*\hat g_2=g_2$.\\

This has the following corollary:
\begin{corollary}\label{iso_preserve}
If $\phi$ is an isometry of $g^0$ which preserves the boundary data, namely 
\begin{eqnarray}
\phi^*\eta(x,t)&=&\eta(x,t)\nonumber\\ 
\left[\phi^*\gamma(x,t)\right]&=&\left[\gamma(x,t)\right],\nonumber
\end{eqnarray}
and $g(t)$ is a solution to the corresponding Ricci flow boundary value problem then $\phi$ is an isometry of $g(t)$ for all $t$.
\end{corollary}
\begin{proof}
It is a consequence of the diffeomorphism invariance of the Ricci flow equation and the uniqueness. If $\phi$ is an isometry of $g^0$ which preserves the boundary data and $g(t)$ is a solution of the Ricci flow boundary value problem then $\phi^* g(t)$ is also a solution with the same initial and boundary data. By uniqueness we obtain that $\phi^*g(t)=g(t)$.
\end{proof}

\subsection{A generalization.}
The methods used in the preceding sections can be applied to prove the following generalization of Theorem  \ref{Ricci_flow}, in which the mean curvature at any $t$ depends on the induced metric on $\partial M$ via a given smooth function $\eta(x,t,g^T,(g^T)^{-1})$.
\begin{theorem}\label{general}
Theorem \ref{Ricci_flow} holds if we replace the boundary condtion for the mean curvature with
\be \mathcal H(g)=\eta\left(x,t,g^T,(g^T)^{-1}\right).\label{new} \ee
The existence time is controlled from below in terms of (\ref{ex1})-(\ref{ex5}), but with the bound on $|\mathcal H(g^0)|$ replaced by a bound on $|\eta|_{C^2}$.
\end{theorem}
\begin{proof}
Regarding the short-time existence of the Ricci-DeTurck flow, estimates on the line of Corollary A.3 of  \cite{MR1163209} establish that the estimates of Lemmata \ref{selfmap} and \ref{lipconst}  remain valid when
\begin{eqnarray}
\widehat G_w&=&\mathcal H'_{g^0}(w-g^0)-(\mathcal H(w)-\mathcal H(g^0))+\nonumber\\
&&\eta(x,t,w^T,(w^T)^{-1})-\eta(x,0,g^{0,T},(g^{0,T})^{-1}) \nonumber\\
G_{w_1}-G_{w_2} &=&\mathcal H'_{g^0}(w_1-w_2)-(\mathcal H(w_1)-\mathcal H(w_2))+\nonumber\\&&\eta(x,t,w_1^T,(w_1^T)^{-1})-\eta(x,0,w_2^{0,T},(w_2^{0,T})^{-1})\nonumber
\end{eqnarray}
with the corresponding constants now controlled by the norm of $\eta(x,t,\cdot,\cdot)$.

The regularity theorems are still valid since the dependence of $\eta$ on $g^T$ is of zero order. Now, pulling back by the DeTurck diffeomorphisms we obtain a solution to the Ricci flow satisfying (\ref{new}). Finally, the arguments  in Sections 4 and 5 also establish the uniqueness for the Ricci-DeTurck  and the Ricci flow boundary value problem in this case.
\end{proof}

\section{On the existence time of the Ricci flow.}\label{example}

It is well known that on closed manifolds the norm of the curvature tensor of the initial metric controls the existence time of the Ricci flow from below. With this in mind, the geometric data required by Theorem \ref{Ricci_flow} seem quite strong and it would be interesting to understand if they can be relaxed.

 In this section we demonstrate that the lower bound on the boundary injectivity radius $i_{b,g^0}$ is in fact necessary, constructing appropriate examples.

 On $M=[0,1]\times S^1\times S^1$ consider  metrics of the form 
\begin{eqnarray}
g=\phi(x)^2dx^2+\psi(x)^2(dy^2+dz^2),\label{ansatz}
\end{eqnarray} 
where $x\in [0,1]$. Set $H:=2 \partial_s(\log \psi)$, which under Ricci flow evolves according to 
$$\partial_t H=\partial^2_s H+H \partial_s H-H^3.$$
In the above we denote $\partial_s=\phi^{-1}\partial_x$. Also, the distance $\mathcal L$ between the two boundary components evolves according to 
$$\frac{d\mathcal L}{dt}=H|_0^{\mathcal L}+\frac{1}{2}\int_0^{\mathcal L}H^2 ds.$$


Now, consider the flat metric $g_0=\varepsilon^2dx^2+dy^2+dz^2$ on $M$, and consider a Ricci flow with initial condition $g_0$, preserving (\ref{ansatz}) and satisfying 
\begin{eqnarray}
H(0,t)=-H(\mathcal L,t)=t^2. \nonumber
\end{eqnarray}
This corresponds to a Ricci flow starting from a $g_0$, with the mean curvature of the boundaries (with respect to the outward direction) equal to $-t^2$ and the conformal class fixed. Note that a first order compatibility condition holds for the mean curvature, hence we can assume that this flow is $C^3$. By Corollary  \ref{iso_preserve} such flow will preserve (\ref{ansatz}).\\


Assuming $t\leq 1$, the maximum principle shows that $|H|\leq 1$. Therefore,
$$\frac{d\mathcal L}{dt}=-2t^2+\frac{1}{2}\int_0^{\mathcal L} H^2 ds\leq -2t^2+\frac{1}{2}\mathcal L.$$
Solving this differential inequality we obtain, for each $\tau\leq 1$,
$$e^{-\frac{1}{2}\tau} \mathcal L(\tau)\leq \varepsilon -2\int_0^\tau t^2e^{-\frac{1}{2}t}dt.$$
Then,  for any $\varepsilon>0$ small there exists a time $T=T(\varepsilon)>0$ such that $\varepsilon -2\int_0^T t^2 e^{-\frac{1}{2}t}dt=0$, hence $\mathcal L(T)=0$. Therefore $T(\varepsilon)$ is an upper bound on the existence time and clearly  $\lim_{\varepsilon\rightarrow 0}T(\varepsilon)=0$.

\section{An extension condition for the Ricci flow.}\label{ext}

In this section we prove Theorem \ref{extend}. 
\begin{proof}[Proof of Theorem \ref{extend}]
Assume that $T<\infty$ and for some $K>0$ 
\begin{equation}\sup_{x\in M} |\Rm(g(t))|_{g(t)}+\sup_{x\in \partial M}|\mathcal A(g(t))|_{g(t)} \leq K \label{blow}\end{equation} 
for all $t<T$. This bound implies that $g(t)$ are uniformly equivalent and in addition that $g^T(t)$ have bounded curvature for $t<T$. 


Next, we observe that the interior injectivity radius $i_{g(t)}$, the injectivity radius of the boundary $i_{g^T(t)}$ and the ``boundary injectivity radius" $i_{b,g(t)}$ are uniformly bounded below for $t<T$. 

Since $g(t)$ are uniformly equivalent, for any $p\in M^o$ there exists a $r_0>0$ such that $dist_t(p,\partial M)\geq r_0$ for all $t<T$. This also shows that the volume ratio
$$\frac{Vol_t(B_t(p,r))}{r^{n+1}}\geq c$$
for all $r\leq r_0$ and $t<T$, which together with the curvature bound gives that  the injectivity radius at $p$ is bounded below. A similar argument controls the injectivity radius of the boundary. 

Moreover,  by comparison geometry the bounds on the curvature and the second fundamental form control the ``focal" distance of the boundary. Then, since the metrics are uniformly equivalent, the boundary cannot form ``self-intersections", hence the boundary injectivity radius $i_b$ is also bounded below.

Now, take a sequence $t_j\nearrow T$. Since $g(t)$ is a smooth flow (up to the boundary), $g(t_j)$ satisfy the compatibility conditions of any order with respect to $\gamma$ and $\eta$. Hence, by Theorem \ref{Ricci_flow} and the bounds above, there exist Ricci flows $h_j(t)$ satisfying the boundary conditions,  $h_j(t_j)=g(t_j)$ and existing for a uniform amount of time. By uniqueness, this implies that $h(t)=g(t)$ for $t\geq t_j$. However, for large $j$ this process constructs an extension of $g(t)$ past time $T$, which is a contradiction.

\end{proof}

\bibliographystyle{amsplain}
\bibliography{boundary_rfmwb}
\end{document}